\documentclass[12pt]{amsart}
\usepackage{amsmath,amsfonts,amssymb,amsthm}
\usepackage{graphicx,color}
\usepackage{epsfig, tikz}
\usepackage{url}
\DeclareMathRadical{\sqrtsign}{symbols}{"70}{largesymbols}{"70}
\newcommand{\bb}{\mathbb}

\newcommand{\integers}{{\bb Z}}
\newcommand{\natls}{{\bb N}}

\newcommand{\reals}{{\bb R}}

\newlength{\figboxwidth}             
\newcommand{\appendixmode}{
        \setcounter{section}{0}
        \renewcommand{\thesection}{\Alph{section}}
        \renewcommand{\thesubsection}{\Alph{section}.\arabic{subsection}}
}

\renewcommand{\bold}[1]{\medskip \noindent {\bf #1 }\nopagebreak}

\newcommand{\cross}{\times}

\newcommand{\st}{\;\: : \;\:}         

\newcommand{\zed}{\integers}

\def\@ifundefined#1#2#3%
  {\expandafter\ifx\csname#1\endcsname\relax#2\else#3\fi}

\@ifundefined{theoremstyle}{
}{
\theoremstyle{plain} 
}
\newtheorem{theorem}{Theorem}[section]
\newtheorem{prop}[theorem]{Proposition}
\newtheorem{proposition}[theorem]{Proposition}
\newtheorem{lemma}[theorem]{Lemma}
\newtheorem{cor}[theorem]{Corollary}
\newtheorem{corollary}[theorem]{Corollary}
\newtheorem{claim}[theorem]{Claim}

\newtheorem{conj}[theorem]{Conjecture}

\@ifundefined{theoremstyle}{
}{
\theoremstyle{definition} 
}
\newtheorem{definition}[theorem]{Definition}
\newtheorem{question}[theorem]{Question}

\theoremstyle{remark}
\newtheorem{remark}[theorem]{Remark}
\newtheorem*{remark*}{Remark}

\newcommand{\cM}{{\mathcal M}}

\mathchardef\GG="321D
\newcommand{\mc}[1]{{}}  
\newcommand{\mcc}[1]{{}}

\newcommand{\FP}[1]{\langle\langle #1 \rangle\rangle}

\numberwithin{equation}{section}

\newcommand{\fp}[1]{\| #1\|}

\begin{document}\title[M\"{o}bius disjointness for 3-IETs]{Mobius disjointness for interval exchange transformations on three intervals}

\author[J. Chaika]{Jon Chaika}
\author[A. Eskin]{Alex Eskin}\maketitle
\begin{abstract}
We show that Sarnak's conjecture on Mobius disjointness holds for interval exchange transformations  on three intervals (3-IETs) that satisfy a mild diophantine condition.
\end{abstract}
\section{Introduction}
Let $\mu: \natls \to \{-1, 0, 1\}$ denote the M\"obius function,
namely, $\mu(n) = 0$ if $n$ is not square-free, $\mu(n) = 1$ if $n$
is square-free and has an even number of prime factors, and $\mu(n) = -1$
if $n$ is square-free and has an odd number of prime factors. 

Let $X$ be a topological space, and let $T: X \to X$ be an invertible
map. We think of the map $T$ as a dynamical system. Peter Sarnak made
the following far-reaching conjecture:

\begin{conj}[M\"obius Disjointness]
\label{conj:Sarnak}
Suppose the topological entropy of $T$ is $0$. Then, for any $x \in
X$, and any (continuous) function $f: X \to \reals$,  
\begin{equation}
\label{eq:Mobius:disjointness}
\lim_{N \to \infty} \frac{1}{N} \sum_{n=0}^{N-1} f(T^n x) \mu(n) = 0. 
\end{equation}
\end{conj}

\begin{definition}[IET]
\label{def:IET}
An interval exchange transformation (IET) is given by a vector
$\vec{\ell} = (\ell_1, \dots, \ell_d) \in \reals^d_+$ and a permutation $\pi$ on
$\{ 1, \dots, d \}$. From $\vec{\ell}$ we obtain
$d$-subintervals of $[0, \sum_{i=1}^d \ell_i)$ as follows:
\begin{displaymath}
I_1 = [0,\ell_1), \quad I_2 = [\ell_1, \ell_1 + \ell_2), \dots, I_d =
[\sum_{i=1}^{d-1} \ell_d, \sum_{i=1}^d \ell_d). 
\end{displaymath}
 Now we obtain a $d$-Interval Exchange Transformation $T = T_{\pi,
   \vec{\ell}} :[0, \sum_{i=1}^d \ell_i) \to [0, \sum_{i=1}^d \ell_i)$
 which exchanges the intervals according to $\pi$. More precisely, if
 $x \in I_j$, then 
\begin{displaymath}
T(x) = x - \sum_{k < j} \ell_k + \sum_{\pi(k) < \pi(j)} \ell_k.
\end{displaymath}
\end{definition}

It is well known that the topological entropy of 
any interval exchange transformation is $0$. Thus,  if
Conjecture~\ref{conj:Sarnak} is true, then
(\ref{eq:Mobius:disjointness}) should hold for any interval exchange
transformation. 

In this paper, we consider only the case $d=3$. Extending our results
e.g.\ to $d=4$ will require fundamental new ideas. 

\begin{lemma}
\label{lemma:3iet:to:rotation}
If $T$ is a $3$-IET with permutation $\begin{pmatrix}1&2&3\\ 3 & 2 & 1\end{pmatrix}$, then $T$ is also the induced map of a rotation on
an interval. 
\end{lemma}

\begin{proof}
Let $\hat{R}: [0, \ell_1 + 2\ell_2 + \ell_3) \to [0,\ell_1 + 2 \ell_2 +
\ell_3)$ be given by
\begin{displaymath}
\hat{R}(x) = \begin{cases}
x + \ell_2 + \ell_3 & \text{if $x \le \ell_1 + \ell_2$} \\
x + \ell_2 + \ell_3 - (\ell_1 + 2 \ell_2 + \ell_3) & \text{otherwise.}
\end{cases}
\end{displaymath}
Then $\hat{R}$ is a $2$-IET (hence a rotation), and the induced map of $\hat{R}$
on the interval $[0,\ell_1+\ell_2+\ell_3)$ is $T$. 
\end{proof}

\bold{The maps $T$ and $R$, the number $\alpha$ and the interval $J$.} 
Let $R: [0,1) \to [0,1)$ denote
rotation by $\alpha$ (i.e. $R(x) = x+\alpha$ mod $1$). Let $J =
[0,z)$ be a subset of $[0,1]$. In the rest of the paper, we assume
that the $3$-IET $T$ is the induced map of $R$ to $J$ and that $x \notin J$ implies $Rx \in J$.

\bold{The numbers $a_k$, $p_k$ and $q_k$.}
Let $a_0, a_1, \dots, $ denote the continued fraction expansion of
$\alpha$. 
Let $p_k/q_k$ denote the continued fraction convergents of
$\alpha$. Then, 
\begin{displaymath}
q_{k+1} = a_{k+1} q_k + q_{k-1}.
\end{displaymath}

\bold{Connection to tori and tori with marked points:} 
Let $\cM_1$ denote the space of flat tori of area $1$. The space
$\cM_1$ admits a transitive action by the Lie group
$SL(2,\reals)$. Let $\hat{Y} \in \cM_1$ denote the square torus. Then,
the stabilizer of $\hat{Y}$ is $SL(2,\zed)$, and thus $\cM_1$ can we
identified with $SL(2,\reals)/SL(2,\zed)$. Under this identification,
a torus with a fundamental domain the parallelogram whose vertices are
the points $0$, $v_1$, $v_2$ and $v_1 + v_2$ 
corresponds to the coset $M SL(2,\zed)$ where $M \in
SL(2,\reals)$ is the matrix whose columns are $v_1$ and $v_2$. The
$SL(2,\reals)$ action on $\cM_1$ coincides with the left
multiplication action on $SL(2,\reals)/SL(2,\zed)$. 

Let $\cM_{1,2}$ denote the space of tori with two marked points. This
space also admits an action by $SL(2,\reals)$. If $g \in SL(2,\reals)$
and $X \in \cM_{1,2}$ is
the torus with fundamental domain the parallelogram with vertices $0$,
$v_1$, $v_2$ and $v_1 + v_2$ and with the marked points $p_1$, $p_2$,
then $g X$ is the torus with fundamental domain the parallelogram with
vertices $0$, $g v_1$, $g v_2$ and $g(v_1 + v_2)$, and with the marked
points $g p_1$ and $g p_2$.  

Recall that $R:[0,1] \to [0,1]$ denotes the rotation by $\alpha$.
Let $\hat{X}=\begin{pmatrix} 1&-\alpha\\0&1 \end{pmatrix} \hat{Y} \in
\cM_1$. Observe
that the first return of the vertical flow on $\hat{X}$ to the
horizontal side coincides with $R$. If $T$ is a 3-IET
given by the induced map of $R$ to an interval
$J=[0,z)$ then $T$ is also the first return of the vertical flow on
$\hat{X}$ to a horizontal segment of length $|J|$. Let $X$ denote the
torus $\hat{X}$ with two marked points, one at each endpoint of the
horizontal segment of length $J$. 

Let $g_{t} = \begin{pmatrix} e^{t} & 0 \\ 0 & e^{-t} \end{pmatrix} \in
SL(2,\reals)$. We refer to the action of the $1$-parameter subgroup
$g_t$ as the geodesic flow on $\cM_1$ (or $\cM_{1,2}$). The action of
$g_t$ on both $\cM_1$ and $\cM_{1,2}$ is ergodic.

\bold{Renormalization.}

We will need to put a diophantine condition on the IET $T$. In terms
of $X \in \cM_{1,2}$, we want the geodesic ray $\{ g_t X \st t > 0 \}$
to spend significant time 
in compact subsets of $\cM_{1,2}$. Directly in terms of the IET data,
our conditions are the following:

ASSUMPTIONS: 
There exist constants $C_1,C_2,C_3,C_4 >1$ such that the
following holds: Suppose $\ell \in \natls$ and  $0<\eta$ are small
enough. Then there exists $\hat{c}_\eta>0$ and $c_\ell$ (depending on
$\eta$ and $\ell$ respectively) so that for every $0<c<c_\ell$ there
exists a constant $k_c \in \natls$ and infinite sequences $L_i,k_i$ so that:
\begin{enumerate}
\item[(A0)] $q_{k_i-1}<c{L_i}<q_{k_i}$.
\item[(A1)] $a_{k_i}<C_1$,
\item[(A2)] $a_{k_i+1}<C_2$ and
\item[(A3)] $a_{k_i+2}<C_3$.
\item[(A4)]  The shortest vertical trajectory on the torus from one
  marked point to a $\fp{q_k\alpha}$ neighborhood of the other has
  length at least $\frac{q_k}{C_4}$. 
\label{Separated}
\item[(A5)]  There exists $u_i$ so that either 
 $$\lambda(\psi_{L_i}^{-1}(u_i))>\hat{c}_{\eta} \text{ and } \lambda(\psi_{L_i}^{-1}((-\infty,u_i)))<\eta \hat{c}_\eta$$ or 
 $$\lambda(\psi_{L_i}^{-1}(u_i))>\hat{c}_\eta \text{ and } \lambda(\psi_{L_i}^{-1}((u_i,\infty)))<\eta \hat{c}_\eta$$  
  where $\psi_r(x) = \sum_{\ell=0}^{r -1} \chi_J(R^\ell x)$ and our 3-IET is the first return map of $R$ to $J$. 
\item[(A6)] $\underset{i \to \infty}{\lim}\, \int_0^1d(R^{L_i}x,x)d\lambda=0$. 
\item[(A7)] We have $L_i>{q_{k_{i}+\ell}}$.
\item[(A8)] We have $\max \{j:\psi_{L_i}^{-1}(j)\neq \emptyset\}- \min \{j:\psi_{L_i}^{-1}(j)\neq \emptyset\}<k_c.$  
\item[(A9)] There exists $v_i$ so that $q_{v_i}\leq L_i<q_{v_i+1}$ and either $L_i=q_{v_i}$ or $a_{v_i+1}>4$ and $L_i=pq_{v_i}$ where $p\leq \lfloor \frac{a_{v_i+1}}4\rfloor$.
\end{enumerate}

Our main result is the following:
\begin{theorem}
\label{theorem:3iet:mobius:disjoint}
Suppose $T$ is a $3$-IET satisfying the assumptions (A0)-(A9). Then,
M\"obius disjointness, i.e. (\ref{eq:Mobius:disjointness}) holds. 
\end{theorem}

Assumptions (A0)-(A9) are reasonable in view of the following:
\begin{prop}
\label{prop:good:assump}
Let $X \in \mathcal{M}_{1,2}$. Let $\nu_T$ be the measure on $\mathcal{M}_{1,2}$ given by $\int f d\nu_T=\frac 1 T\int_0^T f(g_tX)dt$ for all $f \in \mathcal{C}_c(\mathcal{M}_{1,2})$. If there exists a weak-* limit of $\nu_T$ that is not the zero measure then the corresponding 3-IET satisfies assumptions (A0)-(A9). 
\end{prop}

Proposition~\ref{prop:good:assump} is proved in \S\ref{sec:renorm}.

From the Birkhoff ergodic theorem and Proposition~\ref{prop:good:assump}
it is clear that almost all $3$-IET's satisfy the assumptions (A0)-(A9). 
Thus, an immediate corollary of 
Theorem~\ref{theorem:3iet:mobius:disjoint} is the following:

\begin{corollary}
\label{cor:almost:all:mobius}
For almost all 3-IET's, M\"obius disjointness
(i.e. (\ref{eq:Mobius:disjointness}) holds. 
\end{corollary}

\medskip

\bold{Disjointness.}
As in e.g. \cite{Borgain:Sarnak:Ziegler} we derive the M\"obius
disjointness result from a result about joinings of powers of $T$. In
fact, we prove the following:

\begin{theorem}
\label{theorem:disjoint:powers}
If $T$ is a 3-IET that satisfies assumptions (A0)-(A9)
then there exists $\kappa>1$ so that for all
$n>0$, $B_n=\{m<n:T^m \text{ is not disjoint from }T^n\}$  has the
property that $m_1<m_2 \in B_n$ then $\frac{m_2}{m_1}>\kappa$.
\end{theorem}

See Appendix~\ref{sec:appendix:A} for a proof that
Theorem~\ref{theorem:disjoint:powers} implies
Theorem~\ref{theorem:3iet:mobius:disjoint}. This is a straightforward
modification of a note of Harper \cite{Harper:note}.  It is included for completeness.

\begin{remark} In Appendix~\ref{sec:appendix:B}, we prove that for
  almost every 3-IET, $T$,  $T^n$ is disjoint from $T^m$ for all
  $0<n<m$. This gives an alternative (and much easier) proof of
  Corollary~\ref{cor:almost:all:mobius}. However, the  proof in
  Appendix~\ref{sec:appendix:B} does not give a useful diophantine
  condition under which M\"obius disjointness holds. 
\end{remark}

\bold{Related work:} M\"{o}bius disjointness has been shown for a variety of systems see for example \cite{ELR}, \cite{GT} and \cite{W} among others. Most closely related to this work is  \cite{D}, where Vinogradov's circle method is used to prove that every rotation (2-IET) is disjoint from M\"{o}bius;  \cite{Bou} which shows a set of 3-IETs satisfying a certain measure 0 condition are disjoint from M\"{o}bius and \cite{Borgain:Sarnak:Ziegler} where a slightly stronger version of our criterion is introduced to show that the time 1 map of horocycle flows are disjoint from M\"{o}bius. This last paper motivated our approach.

\bold{Further questions and conjectures.}
\begin{question} What is the Hausdorff codimension of the set of $X\in \mathcal{M}_{1,2}$ so that any weak-* limit point of $\nu_T$ is the zero measure?
\end{question}

\begin{conj} For almost every IET that is not of rotation type
  and $n<m\in \mathbb{N}$ we have $T^n$ is disjoint from $T^m$. In
  fact if $U_T$ is the unitary operator associated to (composition
  with) $T$ on $L^2$ function of inttegral zero then there is a
  sequence $k_1,...$ so that $U_{T^{nk_i}}$ converges to the 0
  operator in the weak operator topology and $U_{T^{mk_i}}$ converges
  to the identity operator in the strong operator topology. 
\end{conj}

\bold{Outline:} The Section 2 establishes an abstract disjointness criterion, Proposition \ref{prop:fact:qfact}. Sections 4 uses this to prove Theorem \ref{theorem:disjoint:powers}. Section 3 recalls standard facts about rotations used in Section 4. Section 5 proves Proposition \ref{prop:good:assump}. Appendix A proves that Theorem \ref{theorem:disjoint:powers} implies Theorem~\ref{theorem:3iet:mobius:disjoint}. Appendix B proves that almost every 3-IET has the property that all of its distinct positive powers are disjoint.

  \bold{Acknowledgments:}
J. C. was supported in part by NSF grants DMS-135500 and DMS-1452762 and the Sloan foundation. A. E. is supported in part by NSF grant DMS 1201422 and the Simons Foundation. The authors thank Adam Harper for graciously letting us modify his note and a helpful correspondence. A. E. thanks Princeton University and the Institute for Advanced Study for support during part of this work. The authors would like to thank the Isaac Newton Institute for Mathematical Sciences, Cambridge, for support and hospitality during the program Dynamics of Group Actions and Number Theory  where work on this paper was undertaken. This work was supported by EPSRC grant no EP/K032208/1. 

\section{Disjointness criterion}
\label{sec:criterion}

Let $(X,d)$ be a metric space. We set $X_1 = X_2 = X$, and
write the product $X \cross X$ as
$X_1 \cross X_2$. Let $\lambda$ be a measure on $X$, and 
let $T_1: X_1 \cross X_1$ and $T_2: X_2 \to X_2$ be
$\lambda$-preserving maps. Let $\sigma$ be a joining of $(X_1,T_1,
\lambda)$ and $(X_2, T_2,\lambda)$, i.e.\ $\sigma$ is an ergodic $T_1 \cross
T_2$-invariant measure on $X_1 \cross X_2$ which projects to $\lambda$
in either factor.  

Our basic strategy is due to Ratner \cite{Ratner:Joinings}. In
fact, we use the following proposition:

\begin{proposition}
\label{prop:Ratner:disjointness}
Suppose $S: X \to X$ is a $\lambda$-preserving map which commutes with
$T_1$ and $T_2$. Suppose $d_1 \ge 0$, $d_2 > 0$, and for every 
$\delta > 0$ and any compact set $K \subset X_1 \cross
X_2$ with $\sigma(K) > 1-\delta$ and for every
$\delta > \epsilon > 0$ there exist points $(x,y) \in X_1 \cross X_2$,
$(x', y') \in X_1 \cross X_2$ and $r \in \natls$ so that
the following conditions hold:
\begin{itemize}
\item[{\rm (a)}] $(T_1\cross T_2)^r(x',y') \in K$.  
\item[{\rm (b)}] $(T_1\cross T_2)^r(x,y) \in K$.  
\item[{\rm (c)}] 
$d(T_1^{r} x', S^{d_1} T_1^{r} x) + 
d(T_2^{r} y', S^{-d_2} T_2^{r} y) < \epsilon$.   
\end{itemize}
Then, $\sigma$ is $S^{d_1} \cross S^{-d_2}$-invariant. 
\end{proposition}

\begin{proof}
Suppose $\sigma$ is not $S^{d_1} \cross S^{-d_2}$-invariant. Then 
$(S^{d_1} \cross S^{-d_2})\sigma$ is an ergodic $T_1 \cross
T_2$-invariant measure which is distinct from $\sigma$. Thus, 
$(S^{d_1} \cross S^{-d_2})\sigma$ and $\sigma$ are mutually singular. 
It follows that for any $\delta > 0$ there exists a compact set $K$
with $\sigma(K) > 1-\delta$ such that $(S^{d_1} \cross S^{-d_2}) K
\cap K = \emptyset$. Then there exists $\epsilon$ such that
\begin{displaymath}
d((S^{d_1} \cross S^{-d_2}) K, K ) > \epsilon. 
\end{displaymath}
This is not consistent with conditions (a)-(c). 
\end{proof}

\begin{prop}
\label{prop:fact:qfact} 
Suppose $S$ is continuous except for finitely many points, and suppose
$\lambda$ gives zero measure to the points of discontinuity of $S$. 
Assume 
\begin{enumerate}
\item There exists a sequence of measurable partitions of $X_1$, $U^{(i)}_{-k},...,U^{(i)}_k$ and a sequence of numbers $r_i$ so that 
$$\underset{i \to \infty}{\lim}\,
\int_{U_j^{(i)}}d(T_1^{r_i}(x),S^jx)\, d\lambda(x)=0.$$
\item There exists $\ell \in \{1,2,3\}$, a sequence of measurable sets  $A_i$, and functions $F_i$ preserving the measure $\lambda$, so that
\begin{enumerate}
\item $\underset{i \to \infty}{\lim}\int_{A_i}d(F_iy,y)\,d\lambda(y)=0$.
\item $\underset{i \to
    \infty}{\lim}\int_{A_i}d(S^{-\ell}(T_2^{r_i}y),T_2^{r_i}F_iy)\, d\lambda(y)=0$.
     \\
    This implies
\item $\underset{i \to \infty}{\lim}\int_{F_i(A_i)}d(F_i^{-1}y,y)\,d\lambda(y)=0$.
\item $\underset{i \to
    \infty}{\lim}\int_{F_i(A_i)}d(S^{\ell}(T_2^{r_i}y),T_2^{r_i}F_i^{-1}y)\, d\lambda(y)=0$.
    \mcc{I think the powers may be reversed. Maybe not}
\end{enumerate}
\item There exists an absolute constant $\delta_0 > 0$ such that the
  following holds: for any $0 < \delta < \delta_0$, either 
there exists $a \in \zed$ 
so that for infinitely many $i$,
\begin{equation}
\label{eq:a:is:good}
\sigma(\{(x,y):x\in U_{a}^{(i)} \text{ and } y \in A_i\})> 27 \delta +
14 \lambda( \bigcup_{\ell < a} U_\ell^{(i)})
\end{equation}
or there exists $a' \in \zed$ so that
for infinitely many $i$,
\begin{equation}
\label{eq:a:prime:is:good}
\sigma(\{(x,y):x\in U_{a'}^{(i)} \text{ and } y \in F_i(A_i)\})> 27
\delta + 14 \lambda(\bigcup_{\ell>a'}U_\ell^{(i)}).
\end{equation}
\end{enumerate}



Under the assumptions (1)-(3) with 
(\ref{eq:a:is:good}) there exists
 $d\geq 0$ so that $\sigma$ is $S^d\times S^{-\ell}$ invariant. Also,
 under the assumptions (1)-(3) with 
 (\ref{eq:a:prime:is:good}) there exists
$d\leq 0$ so that $\sigma$ is $S^d\times S^{\ell}$ invariant.
\end{prop}


\begin{lemma}
\label{lemma:find:friend}
Suppose $\epsilon' > 0$ and $\delta > 0$.  Then, for 
any compact set $K \subset X_1 \cross
X_2$ with $\sigma(K) > 1-\delta$ and all $i \in \natls$ sufficiently
large, there exists a compact set $K_i'
\subset X_1 \cross X_2$ with $\sigma(K_i') > 1- 7 \delta$ such that for all
$(x,y) \in K_i'$ with $y \in A_i$,
 there exists $x' \in X_1$ with $(x', F_i y)
\in K$ and $d(x', x ) < \epsilon'$. Similarly, for all $(x,y) \in
K_i'$ with $y \in F(A_i)$, there exists $x' \in X_1$ with
$(x',F_i^{-1} y) \in K$ and $d(x',x) < \epsilon'$. \mcc{check second statement}
\end{lemma}

\begin{proof} 
Define a probability measure $\tilde{\sigma}$ on $X_1 \cross X_2$ by
\begin{displaymath}
\tilde{\sigma}(E) = \frac{\sigma(E \cap K)}{\sigma(K)}.
\end{displaymath}
For $y \in X_2$, 
let $\tilde{\sigma}_y$ be the conditional measure of $\tilde{\sigma}$
along $X_1 \cross \{ y\}$.

Let $B(x,\epsilon')$ denote the open ball of radius $\epsilon'$. For $\tilde{\sigma}$-
almost all $(x,y) \in X_1 \cross X_2$,
$\tilde{\sigma}_y(B(x,\epsilon'/2)) > 0$. 
Therefore, there exists $\rho(\epsilon',\delta) > 0$ 
and a set $K_1 \subset K$ with $\tilde{\sigma}(K_1) > 1-\delta$ such that for
all $(x,y) \in K_1$, 
\begin{displaymath}
\tilde{\sigma}_y(B(x,\epsilon'/2)) > \rho(\epsilon',\delta).  
\end{displaymath}
Let $\pi_2: X_1 \cross X_2 \to X_2$ denote projection to the second
factor. Since the function $y \to \tilde{\sigma}_y$ is measurable, by Lusin's theorem there exists a compact set $K_2 \subset X_2$ with
$\pi_2^*(\tilde{\sigma})(K_2) > 1 - \delta$ on which it is uniformly continuous
relative to the Kantorovich-Rubinstein metric, 
\begin{displaymath}
d(\mu,\nu) = \sup_{f} \left|\int_{X_1} f \, d\nu -\int_{X_1} f \, d\nu
\right|,  
\end{displaymath}
where the sup is taken over all $1$-Lipshitz functions $f: X_1 \to
\reals$ with $\sup |f(x)| \le  1$. 
Then, there exists $\delta' > 0$ such that for all $y,y' \in K_2$ with
$d(y',y) < \delta'$ and for all $x \in X_1$ such that $(x,y) \in K_1$
\begin{displaymath}
\tilde{\sigma}_{y'}(B(x,\epsilon')) > \rho(\epsilon',\delta)/2. 
\end{displaymath}
Then, for all
$y,y' \in K_2$, with $d(y,y') < \delta$ and all $x$ with
$(x,y) \in K_1$,
\begin{equation}
\label{eq:sigma:y:prime:B:x:epsilon}
\sigma_{y'}(B(x,\epsilon') \cap K) > 0. 
\end{equation}
We now estimate $\lambda(K_2)$. For $E \subset X_2$, 
\begin{displaymath}
\pi_2^*\tilde{\sigma}(E) = \tilde{\sigma}(\pi_2^{-1}(E)) =
\frac{\sigma(\pi_2^{-1}(E) \cap K)}{\sigma(K)} \le
\frac{\sigma(\pi_2^{-1}(E))}{1-\delta} = \frac{\lambda(E)}{1-\delta}.
\end{displaymath}
Therefore, 
\begin{displaymath}
\lambda(K_2) \ge (1-\delta) \pi_2^*\tilde{\sigma}(K_2) \ge
(1-\delta)^2 \ge 1 - 2 \delta.
\end{displaymath}


By condition (2a) in Proposition~\ref{prop:fact:qfact}, for $i$
sufficiently large, there exists a
compact set $K_{3,i} \subset X_1 \cross X_2$ with $\sigma(K_{3,i}) > 1-\delta$ 
such that for $(x,y) \in K_{3,i}$, $y \in A_i$ and $i$ sufficiently large, 
\begin{equation}
\label{eq:T2Liy:close:to:y}
d(F_i y, y) < \delta'.
\end{equation}
Now let 
\begin{displaymath}
K'_i = (id \times F_i)^{-1}\left(K_1 \cap (X_1 \cross K_2) \right)
\cap K_{3,i} \cap (X_1 \cross K_2).
\end{displaymath}
Then, 
\begin{displaymath}
\sigma(K'_i)> 1 - 7 \delta.
\end{displaymath}
Suppose $(x,y) \in K'_i$, with $y \in A_i$. For large enough $i$,
(\ref{eq:T2Liy:close:to:y}) holds, and also $(x,y) \in K_1 \cross K_2$ and
$F_i y \in K_2$. Thus 
(\ref{eq:sigma:y:prime:B:x:epsilon}) holds (with $y' = F_i
y$). This implies the first statement of the lemma. The proof of the
second statement is identical.
\end{proof}

\begin{proof}[Proof of Proposition~\ref{prop:fact:qfact}] We establish
  the (\ref{eq:a:is:good}) case. The (\ref{eq:a:prime:is:good}) case is analogous. 
The basic strategy is to choose $(x,y) \in U_a^{(i)} \cross A_i$, and apply
Proposition~\ref{prop:Ratner:disjointness} with $r=r_i$ to the points
$(x,y)$ and $(x',F y)$, where $x'$ is as in
Lemma~\ref{lemma:find:friend}. 
We now give the details. 

Suppose $\delta > 0$ and $0 < \epsilon < \delta$ 
are arbitrary. Let $\Delta$ denote the union of
the points of discontinuity of $S^j$, $1 \le j \le k$. There
exists $c_1(\epsilon) > 0$ such that if we let 
\begin{displaymath}
K_0 = \{ x \in X_1 \st d(x,\Delta) > c_1(\epsilon) \} 
\end{displaymath}
then $\lambda(K_0) > 1-\epsilon > 1 - \delta$. Let 
\begin{displaymath}
K_{00} = \{ x \in X_1 \st d(x,\Delta) > c_1(\epsilon)/2 \}.
\end{displaymath}
Since $K_{00}$ is compact and
$S$ is continuous on $K_{00}$, there exists $\epsilon' > 0$ such that if
$x_1, x_2 \in K_{00}$, with $d(x_1,x_2) < \epsilon'$ then for all $1
\le j \le k$, $d(S^j x_1, S^j x_2) < \epsilon/6$. 
Without loss of generality, we may assume that
$\epsilon' < c_1(\epsilon)/2$. Then, we have, for all $1 \le j \le k$, 
\begin{equation}
\label{eq:S:uniformly:cts}
d(S^{j} x, S^{j} x') < \epsilon/6 \qquad\text{if $x \in K_0$ and $d(x',x) <
  \epsilon'$}. 
\end{equation}

Let $a$ be as in Proposition~\ref{prop:fact:qfact} (3). Write
\begin{displaymath}
\gamma = \lambda(\bigcup_{\ell < a} U_\ell^{(i)})
\end{displaymath}
 

We may assume that $i$ is large enough so that there
exists a compact set  
\begin{displaymath}
K_{1b} \subset  X_1 \setminus \bigcup_{\ell < a} U_\ell^{(i)}
\end{displaymath}
with $\lambda(K_{1b}) > 1- 2\gamma$.

In view of assumption (1) of Proposition~\ref{prop:fact:qfact}, there
exists a compact set $K_{1a} \subset X_1$ with $\lambda(K_{1a}) >
1-\delta$ such that 
\begin{displaymath}
d(T_1^{r_i} x, S^j x) < \epsilon/6 \qquad\text{ for $x \in U^{(i)}_j$.}
\end{displaymath}

In view of assumption (2b) of Proposition~\ref{prop:fact:qfact}, there exists a
compact set $K_{2b} \subset X_2$ with $\lambda(K_{2b}) > 1-\delta$
such that for $y \in K_{2b} \cap A_i$ and $i \in \natls$ sufficiently 
large, 
\begin{displaymath}
|S(T_2^{r_i}y)-T_2^{r_i}F_iy| < \frac{\epsilon}{2}.
\end{displaymath}

As in the proof of Proposition \ref{prop:Ratner:disjointness}, let $K$ be a compact set so that $\sigma(K)>1-\delta$ and $(T_1^{d}\times T_2^{-\ell})K$ are compact and disjoint for all $0\leq d\leq k$ and $0\leq \ell\leq 3$. Formally, $K$ may depend on $d$, but without loss of
generality we may assume that the same $K$ works for all $0 \le d
\le k$. 

Let 
\begin{displaymath}
K''_i = ((K_{1a} \cap
K_{1b}) \cross X)\cap (T_1\cross T_2)^{-r_i} K
\end{displaymath}
Note that $\sigma(K''_i) > 1 - 3 \delta - 2\gamma$. 
Let $K'_i$ be as in Lemma~\ref{lemma:find:friend} for $K''_i$ instead
of $K$. We have
\begin{displaymath}
\sigma(K'_i) > 1 - 21 \delta - 14 \gamma.
\end{displaymath}
Let
\begin{displaymath}
\Omega= (T_1\cross T_2)^{-r_i}(K \cap (K_0 \cross K_0)) \cap K'_i  \cap (X_1 \cross
K_{2b}) \cap (K_0 \cross K_0). 
\end{displaymath}
Then, 
\begin{displaymath}
\sigma(\Omega) \ge 1 - 3\delta - (21 \delta + 14 \gamma) - \delta - 2\delta
= 1 - 27 \delta - 14 \gamma.
\end{displaymath}
Let
\begin{displaymath}
G_i = \{(x,y):x\in U_{a}^{(i)} \text{ and } y \in A_i\}.
\end{displaymath}
By (\ref{eq:a:is:good}), $\sigma(\Omega
\cap G_i) > 0$. 

Now let $(x,y)$ be any point in $\Omega \cap G_i$. Then, by
Lemma~\ref{lemma:find:friend}, there exists $x' \in K_{1a} \cap K_{1b}$ with
$(x',F_i y)  \in (T_1\cross T_2)^{-r_i} K$. Because 
\begin{displaymath}
(x',y') = (x', F y) \in ((K_{1a} \cap
K_{1b}) \cross X)\cap (T_1\cross T_2)^{-r_i} K
\end{displaymath}
letting $r=r_i$ conditions (a) and (b) of
Proposition~\ref{prop:Ratner:disjointness} hold. Also, since $x' \in
K_{1b}$, $x' \not\in \bigcup_{\ell < a} U_a^{(i)}$, and thus we may assume
$x' \in U_b^{(i)}$ for some $b \ge a$. Then, since $x' \in
K_{1a}$, 
\begin{displaymath}
d(T_1^{r_i} x', S^{b} x') < \epsilon/6. 
\end{displaymath}
Also, in view of Lemma~\ref{lemma:find:friend}, 
\begin{displaymath}
d(x', x) < \epsilon',  
\end{displaymath}
and since $x \in K_{1a}$, 
\begin{displaymath}
d(T_1^{r_i} x, S^a x) < \frac \epsilon 6. 
\end{displaymath}
We have $x \in K_0$ and $T_1^{r_i} x \in K_0$. Therefore, by (\ref{eq:S:uniformly:cts}), 
\begin{multline*}
d(T_1^{r_i} x', S^{b-a} T_1^{r_i} x) \le d(T_1^{r_i} x', S^{b} x') + d(S^{b}
x', S^{b} x) + d(S^{b} x, S^{b-a} T_1^{r_i} x)  \\ 
< \frac{\epsilon}{3} + \frac{\epsilon}{6} + \frac{\epsilon}{6} = \frac{\epsilon}{2}.
\end{multline*}
Similarly,  
\begin{displaymath}
d(T_2^{r_i} y', S^{-r} T_2^{r_i} y) = d(T_2^{r_i} F_iy, S^{-r} T_2^{r_i} y) <
\frac{\epsilon}{2}.   
\end{displaymath}

Therefore, assumption (c) in
Proposition~\ref{prop:Ratner:disjointness} also holds with $d_1=b-a$ and $d_2=\ell$, and
Proposition~\ref{prop:Ratner:disjointness} can be applied. 
\end{proof}

Let $\lambda$ denote Lebesgue measure on $[0,1]$. 
Let $S: [0,1] \to [0,1]$ be a 3-IET, $T_1=S^n$ and $T_2=S^m$. Let
 $\sigma$ be an ergodic joining of $T_1$ and $T_2$.  

\begin{corollary}
\label{cor:trivial:joining}
If $S$ is weakly mixing and the conditions of Proposition~\ref{prop:fact:qfact} are
satisfied then $\sigma = \lambda \cross \lambda$. 
\end{corollary}
Note that by \cite{BN} all the 3-IETs we consider are weakly mixing.

This corollary uses the following standard result:
\begin{lemma} (See for example \cite[Lemma~6.14]{Rudolph:book}.) 
\label{lemma:rudolph}
If $(X,B,\mu,T)$ is ergodic and $\sigma$ is a joining of $(X,M_1,\mu,S_1)$ and 
$(Y,M_2,\nu_2,S_2)$ that is $T \times id$ invariant then $\sigma =\nu_1\times \nu_2$.
\end{lemma}
We include a proof because the statement in \cite{Rudolph:book} is slightly more specific.
\begin{proof}[Proof of Lemma~\ref{lemma:rudolph}] Given $A \subset Y$ with positive $\nu_2$ measure let $\sigma_A(B)=\sigma (B \times A)$. This is a measure on $X$. Because $\sigma$ has marginals $\mu,\nu_2$ this measure is absolutely continuous with respect to $\mu$. So it has a Radon-Nikodym derivative $f_A$. By our assumption this is a  $T$ invariant function and so it is constant.  This implies any two rectangles with the same dimensions have the same measure and thus $\sigma$ is the product measure.
\end{proof}

\begin{proof}[Proof of Corollary \ref{cor:trivial:joining}] We show only
  the (\ref{eq:a:is:good}) case, since the (\ref{eq:a:prime:is:good})
  case is similar. Because $\sigma$ is $S^{d_1} \times S^{-r}$ invariant
  it is $S^{md_1}\times S^{-m}$ invariant. Using the fact that $\sigma$ is a
  joining of $S^n$ and $S^m$, this implies that $\sigma$ is
  $S^{md_1+nr}\times id$ invariant. Since $S$ is weak mixing and thus
  totally ergodic, $S^{md_1+nr}$ is ergodic and so by Lemma~\ref{lemma:rudolph}, $\sigma=\lambda \cross \lambda$. 
\end{proof}

In \S\ref{sec:facts:rotation}-\S\ref{sec:work} we will show that
Proposition~\ref{prop:fact:qfact} can be applied to prove
Theorem~\ref{theorem:disjoint:powers}.

\section{Facts about rotations}
\label{sec:facts:rotation}
Let $\|q_k\alpha\|=dist(q_k\alpha,\mathbb{Z})$. Let $\lambda$ be Lebesgue measure. 
$d(x,y)=\min\{|x-y|,1-|x-y|\}$

\begin{lemma}
\label{lemma:q:k:plus:1:ak:qk}
$q_{k+1}=a_{k+1}q_k+q_{k-1}$
\end{lemma}
\begin{lemma} 
\label{lemma:good:bound}
$\frac 1 {q_{k+1}+q_k}<\|q_{k}\alpha\|<\frac 1 {a_{k+1}q_k}$
\end{lemma}
See \cite{Khinchin:book}, 4 lines before equation 34.
\begin{lemma}
\label{lemma:orbit:dense} 
$\{R^ix\}_{i=0}^{q_k-1}$ is $2\|q_{k-1}\alpha\|$ dense for all $x$.
\end{lemma}
\begin{proof}Because $R$ is an isometry, it suffices to prove this for $x=0$. 
We approximate $\{R^n(0)\}_{n=0}^{q_k-1}$ by $\{n\frac{p_k}{q_k} \text{ mod 1}\}_{n=0}^{q_k-1}$, a set that is $\frac 1 {q_k}$ dense. 
 Now $R^\ell(0)$ is within $\|q_k\alpha\|=q_k|\alpha-\frac{p_k}{q_k}|$ of $\ell \frac{p_k}{q_k}$ mod 1 for $0\leq \ell\leq q_k$. 
 Since $\|q_k\alpha\|<\|q_{k-1}\alpha|$ and by Lemma \ref{lemma:good:bound} we have $\frac{1}{q_k}<2\|q_{k-1}\alpha\|$. This establishes the lemma.
\end{proof}
\begin{lemma}
\label{lemma:orbit:sep} 
$\{R^ix\}_{i=0}^{q_k-1}$ is $\|q_{k-1}\alpha\|$ separated for all $x$.
\end{lemma}
\begin{lemma}
\label{lemma:ret:time} 
If $x$ is in an interval $I$ of size $\|q_k\alpha\|$ then the return
time of $x$ to $I$ is either $q_{k+1}$ or $q_{k+1}+q_k$. If $k$ is even and $I =
[-\|q_k \alpha \|,0)$ then the return time of $q_{k+1} + q_k$ takes
place on $[-\|q_{k+1} \alpha\|,0)$. If $k$ is odd and $I=[0,\|q_k\alpha\|)$ then the return time 
of $q_{k+1}+q_k$  takes place on $[\|q_k\alpha\|-\|q_{k+1}\alpha\|,\|q_k\alpha\|)$. 
\end{lemma}
\begin{proof}
First, we assume $k+1$ is odd. If $x\in I$ then $R^{q_{k+1}}x=x-\|q_{k+1}\alpha\|$. 
So if $x$ is not in the leftmost $\|q_{k+1}\|$ of $I$ then $R^{q_{k+1}}x\in I$. Otherwise, $R^{q_k}x=x+\|q_\alpha\|$ is on the right of $I$ and within $\|q_{k+1}\alpha|$ of $I$. So $R^{q_{k+1}+q_k}x\in I$. The case of $k+1$ even is similar.
\end{proof}
\begin{lemma}
\label{lemma:DK}(Denjoy-Koksma) 
If $f$ is bounded variation then $|\sum_{i=0}^{q_k-1}f(R^ix)-q_k\int f \,
d\lambda|\leq var(f)$.
\end{lemma}
Note if $f$ is the characteristic function of an interval
$var(f)=2$. This is the only case we use in the sequel and we present
the proof of this case below. A similar argument will be used
to prove the more general Lemma~\ref{lemma:happy:times}.

\begin{proof}[Proof of special case]
Following the paragraph in the introduction `Connection to tori and tori with marked points' 
we want to understand the intersections of a (half open) vertical line
segment of length $q_k$  to a horizontal line segment of length $z$ 
on $\hat{X}$,  see Figure~\ref{fig:original}.
(Indeed, $R^{q_k}$ is given by a vertical trajectory of length $q_k$ and $\sum_{j=0}^{q_k-1}\chi_J(R^jx)$ is given by the intersection of the corresponding vertical trajectory of length $q_k$ with a horizontal trajectory of length $z$.)  This is equivalent to 
understanding the intersections of a vertical segment of length $1$ to
a horizontal line segment of length $q_kz$ on
$g_{\log(q_k)}\hat{X}$. Call these segments $\gamma_1$ and $\gamma_2$
respectively, see Figure~\ref{fig:renorm1}. 
We close up these two curves as pictured in Figure~\ref{fig:renorm2}
using the following observations:\\
\noindent
\begin{enumerate}
\item Any vertical trajectory of length $q_k$ on $\hat{X}$ has that its endpoints differ by a horizontal vector of length at most $\|q_k\alpha\|<\frac 1{a_{k+1}q_k}.$ This implies we can 
close $\gamma_1$ up by a horizontal segment, $\zeta_1$, of length less than $\frac 1 {a_{k+1}}\leq1$. Call the resulting closed curve $\hat{\gamma}_1$. 

\item We may close up $\gamma_2$ by a vertical segment $\zeta_2'$ of
  length at most $1$, union a horizontal segment $\zeta_2$ of length at most $\frac 1 2 $ which is either contained in $\gamma_2$ or disjoint from it. Call the resulting closed curve $\hat{\gamma}_2$. 
\end{enumerate}

Any vertical segment of length 1 on $g_{\log(q_k)}\hat{X}$ is a translate of $\gamma_1$ and so we may close it up so that it is a translate of $\hat{\gamma}_1$. The intersection of any translate of $\hat{\gamma}_1$ with $\hat{\gamma}_2$ is constant (it is a topological invariant of these curves).  So now we study the intersection of translates of $\hat{\gamma}_1$ and $\zeta_2\cup \zeta_2'$. $\gamma_1$ can intersect $\zeta_2$ either $0$ and $1$ times.   $\gamma_1$ does not intersect $\zeta_2'$ and $\zeta_1$ does not intersect $\zeta_2$. Once again  $\zeta_1$ intersects $\zeta_2'$ at most once. To summarize the intersections with $\gamma_2$ of any two translates of $\gamma_1$  differ by at most $2$.
 \end{proof}

\begin{figure}[ht]
    \centering
    \includegraphics[width=0.3\textwidth]{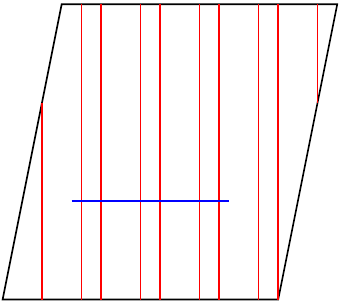}
    \caption{The torus $\hat{X}$. A vertical segment of length $q_k$
      intersects a horizontal slit of length $z$.}
    \label{fig:original}
\end{figure} 

\begin{figure}[ht]
    \centering
    \includegraphics[width=0.3\textwidth]{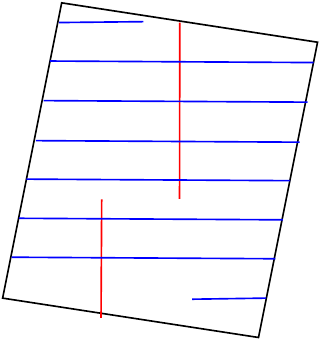}
    \caption{The torus $g_{\log(q_k)} \hat{X}$: A vertical segment
      $\gamma_1$ of length $1$ (drawn in red) intersects a horizontal
      slit $\gamma_2$ of length $q^k z$ (drawn in blue).}
    \label{fig:renorm1}
\end{figure} 

\begin{figure}[ht]
    \centering
    \includegraphics[width=0.3\textwidth]{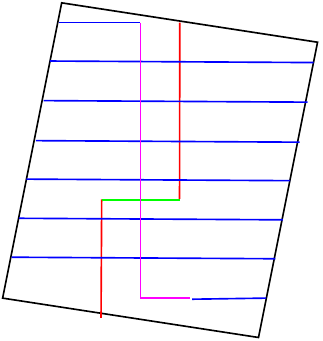}
    \caption{Closing the curves. We complete the vertical segment 
     $\gamma_1$ to a closed
      curve $\hat{\gamma_1}$ by adding a horizontal segment $\zeta_1$
      (drawn in green). Simularly, we close up the horizontal slit
      $\gamma_2$ to obtain a closed curve $\hat{\gamma}_2$ 
      by adding in a horizontal segment $\zeta_2$ and a
      vertical segment $\zeta_2'$ (drawn in purple).}
    \label{fig:renorm2}
\end{figure}

\begin{lemma} 
\label{lemma:happy:times} 
For all $k\in \mathbb{N}$ with $a_{k+1}>4$, and $i \in \mathbb{N}$ with $i \leq \lfloor \frac {a_{k+1}}{4}\rfloor$ we have that there exists $j$ with $\lambda(\psi_{iq_k}^{-1}(j))>\frac 1 {12}$ and either $j-\min\{\ell: \psi_{jq_k}^{-1}(\ell) \neq \emptyset\}\leq 1$ or $\max\{\ell:\psi_{jq_k}^{-1}(\ell)\neq \emptyset\}-j\leq 1$.  Moreover $\psi_{iq_k}$ is at most $i+2$ valued.
\end{lemma}

\begin{proof}
This is similar to the proof of Lemma~\ref{lemma:DK}, 
but the vertical segment on $\hat{X}$ has
length $iq_k$. We once again work on $g_{\log(q_k)}\hat{X}$, where
the vertical segment $\gamma_1$ has length $i$, and the slit
$\gamma_2$ has length $q_k z$ (See Figure~\ref{fig:renorm2}). Thus, we
need to estimate the number of intersections between $\gamma_1$ and
$\gamma_2$.  As in the proof of Lemma~\ref{lemma:DK}, we make the
following observations (see Figure~\ref{fig:renorm2}):
\begin{enumerate}
\item Any vertical trajectory of length $iq_k$ on $\hat{X}$ has that its endpoints differ by a horizontal vector of length at most $i\|q_k\alpha\|<\frac i{a_{k+1}q_k}.$ This implies we can 
close $\gamma_1$ up by a horizontal segment, $\zeta_1$, of length less than $\frac i {a_{k+1}}<\frac 1 4 $. Call the resulting closed curve $\hat{\gamma}_1$. 

\item
We may close up $\gamma_2$ by a vertical segment of length at most $1$, $\zeta_2'$, union a horizontal segment of length at most $\frac 1 2$, $\zeta_2$, which is either contained in $\gamma_2$ or disjoint from it. Call the resulting closed curve $\hat{\gamma}_2$. 
\end{enumerate}

Any vertical segment of length $i$ on $g_{\log(q_k)}\hat{X}$ is a
translate of $\gamma_1$ and so we may close it up so that it is a
translate of $\hat{\gamma}_1$. As in the proof of
Lemma~\ref{lemma:DK}, the intersection of any translate of $\hat{\gamma}_1$ with $\hat{\gamma}_2$ is constant (it is a topological invariant of these curves).  So now we study the intersection of translates of $\hat{\gamma}_1$ and $\zeta_2\cup \zeta_2'$. $\gamma_1$ can intersect $\zeta_2$ between $0$ and $i$ times.   $\gamma_1$ does not intersect $\zeta_2'$ and $\zeta_1$ does not intersect $\zeta_2$. Also $\zeta_1$ intersects $\zeta_2'$ at most once. To summarize the intersections with $\gamma_2$ of any two translates of $\gamma_1$  differ by at most $i+1$.

Observe that on every horizontal line, a segment of at least $\frac{1}{4}$  has that all the corresponding translates of $\hat{\gamma}_1$ for this line segment intersect $\zeta_2$ or all of them do not. (Indeed there is a segment of size at least $\frac 1 2$ so that a vertical segment of length 1 from any point on this segment misses $\zeta_2$ and $\{j\|q_k\alpha\|:0\leq j\leq i\}$ is contained in an interval of length at most $\frac 1 4 $. That is, there is a subinterval of size $\frac 1 4$ so that for each $x$ in this subinterval we have that $j\|q_k\alpha\|+x$ is in the subinterval of size $\frac 1 4 $ for all $0\leq j\leq i$.) So on a subset of this set of measure at least $\frac 1 8$ the translates of $\hat{\gamma}_2$ either all intersect $\zeta_2'$ or all miss $\zeta_2'$. This set satisfies the lemma and it is either within one of the maximal 
or within $1$ of the minimal. 
\end{proof}
 
\begin{lemma}\label{lem:hit interval}  If $\hat{I}$ is an interval of size at least $\gamma \|q_k\alpha\|$ and $q_L>12 \gamma^{-1}\|q_k\alpha\|$ then for all $x$ we have $\frac 1 {q_L}\sum_{j=0}^{q_L-1}\chi_{\hat{I}}(R^j x)\in [\frac 1 2 \lambda(\hat{I}),2 \lambda(\hat{I})]$. Also for all $\gamma>0$ there exists $u$ so that if $\hat{I}$ is an interval of size at least $\gamma \|q_k\alpha\|$ then $\frac 1 {q_L}\sum_{j=0}^{q_L-1}\chi_{\hat{I}}(R^j x)\in [\frac 1 2 \lambda(\hat{I}),2 \lambda(\hat{I})]$ for all $L>k+u$ and  if $t>q_{k+u}$ we have $\frac 1 t \sum_{j=0}^{t-1}\chi_{\hat{I}}(R^jx)\geq \frac 1 4 \lambda(\hat{I})$.
\end{lemma}
This follows by Lemma~\ref{lemma:DK}. Indeed $\sum_{j=0}^{q_L-1}\chi_{\hat{I}}(R^j x)\in  [q_L\lambda(\hat{I})-2,q_L\lambda(\hat{I})+2]$ so if the lemma follows if $q_L\lambda(\hat{I})>4$. Since $\|q_k\alpha\|<\frac 1 {3q_{k+1}}$ this is the case. To see the second claim first notice that $q_{k+2}=q_{k+1}+q_k>2q_k$. So if $2^b>12\gamma^{-1}$ then $q_{L}>12\gamma^{-1}\|q_k\alpha\|$ whenever $L\geq k+2b$, and the second claim follows from the first with $u=2b$. To see the last claim, notice 
$\sum_{j=0}^{t-1}\chi_{\hat{I}}(R^jx) \geq \sum_{i=0}^{\lfloor \frac t {q_{k+u}}\rfloor-1} \sum_{j=0}^{q_{k+u}-1}\chi_{\hat{I}}(R^{j}R^{iq_{k+u}}x)$ and apply the previous sentence to obtain that this is at least $\lfloor \frac t {q_{k+u}}\rfloor q_{k+u}\frac 1 2 \lambda(\hat{I})$. Since $t<2 \lfloor \frac t {q_{k+u}}\rfloor q_{k+u}$ we have the final claim.
\begin{lemma}\label{lem:sum bound} $\underset{k \to \infty}{\lim}\frac{\sum_{i=1}^ka_i}{q_k}=0.$
\end{lemma}
\begin{proof} By Lemma \ref{lemma:q:k:plus:1:ak:qk} we have that $q_\ell>a_{\ell}q_{\ell-1}$ and $q_{\ell+2}>2q_\ell.$ So by induction we have that $q^k>2^{\frac{j}2}\prod_{i=1}^{k}a_i$ where $j=|\{i<k:(a_i,a_{i-1})=(1,1)\}|$. 
\end{proof}
\section{Applying the Criterion}
\label{sec:work}
In this section we show that (A0)-(A9) imply the assumptions of Proposition \ref{prop:fact:qfact}. Proposition \ref{prop:technical:cond3} connects between rotations and powers of 3-IETs. Lemma \ref{lemma:first:sum} is an intermediate step in showing that (A0)-(A9) imply the assumptions of Proposition \ref{prop:technical:cond3} and Section \ref{sec:subsec:proof:of:prop:technical} completes the argument.

In view of Proposition~\ref{prop:fact:qfact} and
Corollary~\ref{cor:trivial:joining}, to prove
Theorem~\ref{theorem:disjoint:powers} it is enough to prove the
following:
\begin{proposition}
\label{prop:satisfy:2a:2b:2c}
Suppose assumptions (A0)-(A9) are satisfied. Then, there exists a
constant $\epsilon_* > 0$ (depending only on the constants in (A0)-(A9))
such that the following holds: Suppose $n \in \natls$, $m'<m<n$, and 
\begin{equation}\label{eq:sublacunary}
\left|\frac{m'}{m} - 1\right| \le \epsilon_*. 
\end{equation}
Then the assumptions of Proposition~\ref{prop:fact:qfact} can be
satisfied for $X_1 = X_2 = J$, 
$T_1 = S^n$ and $T_2$ either $S^m$ or $S^{m'}$. (Note
that we view $T_1$ and $T_2$ as maps from $J$ to $J$). 
\end{proposition}

\bold{Notation.} Before starting the proof of Proposition~\ref{prop:satisfy:2a:2b:2c} we introduce some
notation.  Let 
\begin{displaymath}
\psi_M(x) = \sum_{\ell=0}^{M -1} \chi_J(R^\ell x).
\end{displaymath}
Then, for any $x \in J$ so that $R^Mx\in J$,
\begin{equation}
\label{eq:time:change:general}
S^{\psi_M(x)} x = R^M x.
\end{equation}

\bold{Picking parameters.} Let 
\begin{equation}
\label{eq:def:c}
c = \frac{m - m'}{n}. 
\end{equation}
Since we have $0 < m' < m \le n$, note that
\begin{displaymath}
\frac{m}{m'} = 1 + \frac{c n}{m'} \ge 1 + c.
\end{displaymath}
Thus, in order to prove Theorem~\ref{theorem:disjoint:powers}, we may
assume that $c$ is small. 

Let $k$, $L$ be such that assumptions
(A0)-(A9) are satisfied.
Let $r_k= \lfloor \frac{\lambda(J)L}n \rfloor$. 
Let
\begin{displaymath}
w_k = \lfloor \frac{r_k}{\lambda(J)}\rfloor. 
\end{displaymath}
Then, 
\begin{displaymath}
(m-m') w_k = \frac{(m-m')}{n } L + O(\frac{m-m'}{2 \lambda(J)}) = c L +
O(n). 
\end{displaymath}
Thus, in view of (A0) and (\ref{eq:def:c}), we have
\begin{equation}
\label{eq:m:mprime:wk:range}
(m-m') w_k \in [q_{k-1},q_k).
\end{equation}
\medskip

The proof of Proposition~\ref{prop:satisfy:2a:2b:2c} relies on the
following technical result:
\begin{prop}
\label{prop:technical:cond3} 
There exists $c_0, \tilde{c}>0$, $\hat{C} > 0$  and $u \in \natls$
depending only on the constants
$C_1,...,C_4$ of the assumptions (A0)-(A4) so that
if $c < c_0$ (where $c$ is as in (\ref{eq:def:c})),  
then there exists $\hat{m}\in \{m,m'\}$ and $d \in \{-3,-2,-1,1,2,3\}$
so that, after passing to a subsequence, for all large enough $k$, 
\begin{enumerate}
 \item there exists $\tilde{A} \subset J$ with $\lambda(\tilde{A})>\tilde{c}$ so that for all $y \in \tilde{A}$ we have 
\begin{equation}
\label{eq:Rqk:S:hat:m:rk}
R^{q_k}S^{\hat{m}r_k}y = S^d S^{\hat{m}r_k} R^{q_k}y.
\end{equation}
\item If $N>\frac{q_{k+u}}{2\hat{m}}$ then for any $y \in [0,1]$ we
  have 
\begin{displaymath}
\frac{1}{N} \sum_{i=0}^{N-1}\chi_{\tilde{A}}(S^{i\hat{m}}y) >  \frac{\lambda(\tilde{A})}{\hat{C}} \text{ and } \frac{1}{N} \sum_{i=0}^{N-1}\chi_{R^{q_k}\tilde{A}}(S^{i\hat{m}}y) >  \frac{\lambda(\tilde{A})}{\hat{C}}. 
\end{displaymath}
\end{enumerate}
\end{prop}

\subsection{Proof of Proposition~\ref{prop:satisfy:2a:2b:2c} assuming Proposition~\ref{prop:technical:cond3}} 
We now prove Proposition \ref{prop:satisfy:2a:2b:2c} assuming
Proposition~\ref{prop:technical:cond3}. 
We assume that the constant $c$ in
Proposition~\ref{prop:technical:cond3} is small enough, and $d=1$. 
The case of $d\in\{3,2,-1,-2,-3\}$ is similar.

In view of assumption (A8) there exists an interval
$K_1(L)\subset \mathbb{N}$ of size at most $k_c$ so that for any $x \in J$,
$\psi_L(x) \in K_1(L)$. Since $R(J^c) \subset J$, there exists an
interval $K_2(L)$ of size $(k_c+1)$ such that for all $x \in [0,1]$, 
$\psi_L(x) \in K_2(L)$. 

We have
\begin{displaymath}
\int_0^1 \psi_L(x) \, d\lambda(x) = \lambda(J) L.  
\end{displaymath}
Therefore, $\lfloor \lambda(J) L \rfloor \in K_2(L)$. 
Since $R(J^c) \subset J$, it follows that for all $x \in J$,  
\begin{displaymath}
\lfloor \lambda(J) L \rfloor = \psi_{L'(x)}(x) + \delta(x)
\qquad\text{ where $0 \le \delta(x) \le 1$, and $|L'(x) -
  L| < (k_c+1)$. }
\end{displaymath}
Write
\begin{displaymath}
n r_k = \lfloor \lambda(J) L \rfloor + \epsilon_k, \qquad \text{ where
  $|\epsilon_k| < n$.}
\end{displaymath}
Now, for $x \in J$, by (\ref{eq:time:change:general}), 
\begin{displaymath}
T_1^{r_k} x = S^{n r_k} x = S^{\epsilon_k} S^{\lfloor \lambda(J) L
  \rfloor} x = S^{\epsilon_k} R^{L'(x)} x = S^{\epsilon_k} R^{L'(x) -
  L} R^L x
\end{displaymath}
Thus, by (A6) condition (1) of Proposition~\ref{prop:fact:qfact} follows, 
(with the size of the partition dependent on
$n$). 

Let $F_k$ be the first return map of $R^{q_k}$ to
$J$. (Essentially we want $F_k$ to be $R^{q_k}$, but we want $F_k$ to
be a map from $J$ to $J$). 
Since $R^{q_k}$
tends to the identity map as $k \to \infty$, condition (2a) of
Proposition~\ref{prop:fact:qfact} follows.

For $x \in \hat{A}$, and since we are assuming $d=1$,
(\ref{eq:Rqk:S:hat:m:rk}) becomes
\begin{displaymath}
R^{q_k} T_2^{r_k} x = S T_2^{r_k} R^{q_k} x.
\end{displaymath}
Since $R^{q_k}$ tends to the identity as $k \to \infty$, there exists
a subset $E \subset J$ of almost full measure such that for $x \in E$,
$R^{q_k} x = F_k x$. Then, for $x \in E \cap \hat{A}$, 
\begin{displaymath}
R^{q_k} T_2^{r_k} x = S T_2^{r_k} F_k x.
\end{displaymath}
Condition (2b) of Proposition~\ref{prop:fact:qfact} follows. 

We now begin the proof of Condition (3) of
Proposition~\ref{prop:fact:qfact}. In (A0)-(A9) we choose $\eta <
(96\cdot 25)^{-1}$. Let $\rho = \hat{c}_\eta$, and choose $\delta_0 <
\frac{1}{12} \eta \rho$. Then, by (A5),  
we can either choose $a$ such that for $i$
sufficiently large,  
\begin{equation}
\label{eq:choice:of:a}
\frac{1}{96\hat{C}} \lambda(U_a^{(i)}) > 14 \delta_0 + 27
   \lambda(\bigcup_{\ell < a} U_\ell^{(i)}) \text{ and }
   \lambda(U_a^{(i)}) > \rho,
\end{equation}
or choose $a'$ so that for $i$ sufficiently large, 
\begin{equation}
\label{eq:choice:of:a:prime}
\frac{1}{96\hat{C}} \lambda(U_{a'}^{(i)}) > 14 \delta_0 + 27
   \lambda(\bigcup_{\ell > a'} U_\ell^{(i)}) \text{ and }
   \lambda(U_{a'}^{(i)}) > \rho,
\end{equation}
We need a lemma to obtain Condition (3) of
Proposition~\ref{prop:fact:qfact} from Proposition \ref{prop:technical:cond3}:

\begin{lemma}
\label{lemma:strectch:in}
For every $\rho>0$ there exists $b \in \natls$ so that if for some $s \in
\zed$, $\lambda(\psi_L^{-1}(s))>\rho$ where $L\geq q_\ell$ so that either $L=q_\ell$ or $a_{\ell+1}>4$ and $L=iq_\ell$ for $i\leq \lfloor \frac{a_{\ell+1}}4 \rfloor$ then there exists a measurable set $V$
with the following properties: 
\begin{itemize}
\item We have
$$\lambda\left(\bigcup_{j=0}^{q_{\ell-b}-1}R^j(V)\right)>\frac{1}{2} \lambda(\psi_L^{-1}(s)). $$ 
\item We have $R^j(V) \subset \psi_L^{-1}(s)$ for all $0\leq
  j<q_{\ell-b}$.
\item The sets  $R^j(V)$, $0\leq j<q_{\ell-b}$ are pairwise disjoint.   
\end{itemize}
\end{lemma}

\begin{proof}[Proof of Lemma~\ref{lemma:strectch:in}] 
We claim that there exists $b' \in \natls$ so that 
\begin{equation}
\label{eq:claim:stretch:in}
\lambda(\{x: R^jx\in \psi^{-1}_L(s) \text{ for all }0\leq
j<5q_{j-b'}\})> \frac 4 5\lambda(\psi_L^{-1}(s)). 
\end{equation}
To prove (\ref{eq:claim:stretch:in}), note that for any $x \in [0,1]$ 
there exist at most $2$ different $0\leq j<q_\ell-1$ so that $\psi_L(R^{j+1}x)\neq
\psi_L(R^jx)$. (Indeed the set where $\psi_L(x) \neq \psi_L(Rx)$ is
two intervals of length $i\|q_\ell \alpha\|<\frac i {a_{\ell+1}q_\ell }\leq \frac 1 {4q_\ell}<\|q_{\ell-1}\|$. Any orbit of length $q_\ell-1$ can
hit each of these intervals at most once.)   Thus for any $r \in
\natls$, 
\begin{equation}
\label{eq:tmp:claim:stretch:in}
\lambda(\{x:\psi_L(x)\neq \psi_L(R^jx) \text{ for some }0<j<r\})\leq
\frac{3r}{q_\ell-1}. 
\end{equation}
(Indeed each orbit of length $q_\ell-1$ can have at most 3 consecutive
stretches in this set. These stretches have length at most $r$.) Now
(\ref{eq:tmp:claim:stretch:in}) implies (\ref{eq:claim:stretch:in}). 

Now we build $V$. For each $x \in G= \{x: R^jx\in \psi^{-1}_L(s) \text{ for all }0\leq
j<5q_{\ell-b'}\}$ let $m_x=\max\{j:R^jx\notin \psi_L^{-1}(s)\}+1$ and
$M_x= \min\{j:R^j\notin \psi_L^{-1}(s)\}-1$. Let
\begin{displaymath}
V=\bigcup_{x\in G}\bigcup_{j=0}^{\lfloor \frac{M_x-m_x}{q_{\ell-b}}\rfloor-1} R^{m_x+jq_{\ell-b}}x. 
\end{displaymath}
\end{proof}

\begin{proof}[Proof of Condition (3) of
Proposition~\ref{prop:fact:qfact} continued] We assume that
(\ref{eq:choice:of:a}) holds. (The proof in the case
(\ref{eq:choice:of:a:prime}) holds is virtually identical). Let $a$
and $\rho$ be as in (\ref{eq:choice:of:a}). We then apply
Lemma~\ref{lemma:strectch:in} with this $\rho$ and $s = a$. Then, let
$V$, $\ell$ and $b$ be as in Lemma~\ref{lemma:strectch:in}. 
We assume $\epsilon_*$ (and thus $c$) is small
enough so that (A7) implies that
\begin{equation}
\label{eq:L:prime:big}
q_{\ell - b} >q_{k+u+4}. 
\end{equation}
Let $\sigma$ be any joining of
$S^n \times S^{\hat{m}}$ and for each $x$ let $\Sigma_x$ denote the
points $y$ so that $(x,y)$ is $\sigma$-generic. Let
\begin{displaymath}
E_1 = \bigcup_{i=0}^{n-1} R^i V \cap J.
\end{displaymath}
We are assuming that $x \in V \subset J$, 
and also we are assuming that $Ry \in J$ whenever $y \not\in J$. Then, 
for at least half of $0 \le i < n$ we have $R^i x \in J$, it follows that
\begin{displaymath}
\lambda(E_1) > \frac{n}{2} \lambda(V) > \frac{n}{4 q_{\ell-b}}
\lambda(\psi_L^{-1}(a)).  
\end{displaymath}
We can choose $N \in \natls$ so that $N \ge \frac{q_{\ell-b}}{12n}$ and also 
\begin{equation}
\label{eq:phi:L:Snj:x:a}
\psi_L(S^{nj} x ) = a \quad\text{ for all $x \in E_1$ and all $0 \le j
  < 2N$.}
\end{equation}
Let 
\begin{displaymath}
E_1' = \bigcup_{j=0}^{N-1} S^{nj} E_1. 
\end{displaymath}
Then, in view of (\ref{eq:phi:L:Snj:x:a}),
\begin{equation}
\label{eq:E1:subset:phi:L:inverse:a}
S^{nj} E_1' \subset \psi_L^{-1}(a) \qquad\text{for all $0 \le j < N$.} 
\end{equation}
Note that $R^ix$, $R^jx$ are in distinct $S^n$ orbits if $|i-j|<n$ and
$R^ix,R^jx \in J$. This means that the above union is disjoint, and
thus
\begin{displaymath}
\lambda(E_1') = N \lambda(E_1) > \frac{N n}{4q_{\ell-b}}
\lambda(\psi_L^{-1}(a)) > \frac{1}{48} \lambda(\psi_L^{-1}(a)).  
\end{displaymath}
Since $\sigma$ is a self joining of $\lambda$, we can find $E \subset
E_1 \cross [0,1]$ so that
\begin{equation}
\label{eq:measure:sigma:E}
\sigma(E) > \frac{1}{2} \lambda(E_1) > \frac{1}{96} \lambda(\psi_L^{-1}(a)). 
\end{equation}
We have, by (\ref{eq:L:prime:big}),
\begin{equation}
\label{eq:N:big}
N \ge \frac{q_{\ell-b}}{12 n} > \frac{q_{k+u+4}}{12 n} > \frac{q_{k+u}}{\hat{m}}.
\end{equation}
where $\hat{m} \in \{m,m'\}$ and $u$ is
as in Proposition~\ref{prop:technical:cond3}.
Let $\chi_{\tilde{A}}$ denote the characteristic function of
$\tilde{A}$, we have, in view of (\ref{eq:N:big}) and conclusion (2)
of Proposition \ref{prop:technical:cond3}, for any $y \in [0,1)$,
\begin{equation}
\label{eq:most:in:hatA}
\frac{1}{N} \sum_{j = 0}^{N - 1}
\chi_{\tilde{A}}(S^{\hat{m} j} y) \ge \frac{\lambda(\tilde{A})}{\hat{C}}.
\end{equation}
Then, since $\sigma$ is $S^{n} \cross S^{\hat{m}}$ invariant, 
\begin{align*}
\sigma(\psi_L^{-1}(a) \times \tilde{A}) & = \frac{1}{N} \sum_{j=0}^{N-1} 
\sigma(S^{-n j} \psi_L^{-1}(a) \times S^{-\hat{m} j } \tilde{A})
& \\
& \geq \frac{1}{N} \sum_{j=0}^{N-1} 
\sigma((S^{-n j} \psi_L^{-1}(a) \times S^{-\hat{m} j } \tilde{A}) \cap
E) & \\
& = \int_{E} \left(\frac{1}{N} \sum_{j=0}^{N-1}
\chi_{\tilde{A}}(S^{\hat{m}j} y)\right) \, d\sigma(x,y) & \text{by (\ref{eq:E1:subset:phi:L:inverse:a})}
\\
& \ge \frac{\lambda(\tilde{A}) \sigma(E)}{\hat{C}} & \text{by (\ref{eq:most:in:hatA})}
\\
&  > \frac{1}{96 \hat{C}} \lambda(\psi_L^{-1}(a)) \, \lambda(\tilde{A}), &\text{by (\ref{eq:measure:sigma:E}).} 
\end{align*}
Condition (3) of Proposition~\ref{prop:fact:qfact} now follows
immediately from (\ref{eq:choice:of:a}).  
\end{proof}

\subsection{The main lemma}
The next lemma about rotations is the key step in the proof of Proposition~\ref{prop:technical:cond3}. 
\begin{lemma}
\label{lemma:first:sum}
Assume (A0)-(A4) are satisfied and also that $(m-m')w_k
\in [q_{k-1},q_k)$. Let $C_1, \dots, C_4$ be as in assumptions (A0)-(A4).
Then there exist $c_2 > 0$ and $C' > 0$ depending only on $C_1, \dots,
C_4$ such that for all $k \in \natls$
there exists an interval $I \subset [0,1)$ and a set of natural numbers 
$E=\{e,...,e+c_2q_k\}$ so that 
\begin{enumerate}
\item $|I| \ge C'\fp{q_{k}\alpha}$
\item For all $x \in \bigcup_{i \in E}R^iI$.
\begin{equation}
\label{eq:E:property}
\sum_{\ell=0}^{mw_k-1}(\chi_J(R^\ell x)-\chi_J(R^{\ell+q_k}x))-\sum_{\ell=0}^{m'w_k-1}(\chi_J(R^{\ell}x)-\chi_J(R^{\ell+q_k}x))\in \{-1,1\}. 
\end{equation}
\end{enumerate}
\end{lemma}

In the rest of this subsection, we will prove
Lemma~\ref{lemma:first:sum}. We will derive
Proposition~\ref{prop:technical:cond3} from Lemma~\ref{lemma:first:sum}
in \S\ref{sec:subsec:proof:of:prop:technical}.
 The proof of this lemma is complicated and so we provide a brief sketch: 
We use (\ref{eq:difference:non:zero}) to have a criterion for $\sum_{\ell=0}^{mw_k-1}(\chi_J(R^\ell x)-\chi_J(R^{\ell+q_k}x))-\sum_{\ell=0}^{m'w_k-1}(\chi_J(R^{\ell}x)-\chi_J(R^{\ell+q_k}x))\in \{-1,1\}.$ Claims \ref{claim:hit:both}, \ref{claim:choose:E0} and \ref{claim:E0:works} use this criterion to prove the claim. Claim \ref{claim:hit:both} and subsequent comments identify $I$. Claim 4.6 identifies $E$. Claim 4.7 is used to show the critierion given by (\ref{eq:difference:non:zero}) holds for these $I$ and $E$. 

\begin{proof}[Proof of Lemma \ref{lemma:first:sum}]
Let
\begin{displaymath}
c_2 \leq  \frac{1}{C_4} \qquad\text{ so that $c_2 < 1$.}
\end{displaymath}
Recall that $J=[0,z]$.
Assume $k$ is odd. (This is an assumption of convenience of
exposition. If $k$ is odd then $R^{q_k} {0} =-\|q_k\alpha\|$, if $k$ is even it is $\|q_k\alpha\|$.
 Thus if $k$ is even all sets $[-\|q_k\alpha\|,0)$ should be
 $[0,\|q_k\alpha\|)$ and $[z-\|q_k\alpha\|,z)$ should be
 $[z,z+\|q_k\alpha\|)$).  

Observe 
\begin{multline}
\label{eq:sum:rearrangement}
\sum_{\ell=0}^{mw_k-1}\chi_J(R^\ell x)-\chi_J(R^{\ell+q_k}x)-(\sum_{\ell=0}^{m'w_k-1}\chi_J(R^{\ell}x)-\chi_J(R^{\ell+q_k}x))=\\
\sum_{\ell=0}^{(m-m')w_k-1}\chi_J(R^\ell
R^{m'w_k}x)-\sum_{\ell=0}^{(m-m')w_k-1}\chi_J(R^{\ell+q_k}R^{m'w_k}x))
\equiv F(R^{m' w_k} x),  
\end{multline}
where
\begin{displaymath}
F(y) = \sum_{\ell=0}^{(m-m')w_k-1}\chi_J(R^\ell y)-
\sum_{\ell=0}^{(m-m')w_k-1}\chi_J(R^{\ell+q_k}y)).
\end{displaymath}
Recall that $J=[0,z)$. We have
\begin{equation}
\label{eq:F:alternative}
F(y) = \sum_{\ell=0}^{(m-m')w_k-1}\chi_{[-\|q_k \alpha\|,0)}(R^\ell y)
  - \sum_{\ell=0}^{(m-m')w_k-1}\chi_{[z-\|q_k \alpha\|,z)}(R^\ell y).
\end{equation}
Note that, since $(m-m')w_k<q_k$,  by Lemma \ref{lemma:orbit:sep},
each of the sums in (\ref{eq:F:alternative}) is at most $1$. Thus,
$F(y) \in \{-1, 0, 1\}$, and 
\begin{multline}
\label{eq:difference:non:zero}
F(y) \in \{ 1, -1 \} \\ \text{if and only if $\{R^\ell  y \}_{\ell=0}^{(m-m')w_k}$ 
hits $[-\fp{q_k\alpha},0) \cup [z-\fp{q_k\alpha},z)$ exactly
once. }
\end{multline}

Consider $[-\fp{q_k\alpha},0)$. By Lemma~\ref{lemma:ret:time}, 
the function that assigns to a point in $[-\|q_k\alpha\|,0)$ its first
return time takes two values, $q_{k+1}$ and $q_{k+1}+q_k$. 
The return time of $q_{k+1}+q_k$ occurs on $[-\fp{q_{k+1}\alpha},0)$. 

\begin{claim}
\label{claim:hit:both}
\begin{itemize} 
\item[{\rm (a)}] 
For every $x \in [-\|q_k\alpha\|,0)$ there exists $j \in \{1,...,q_{k+1}+q_k\}$ so that 
$$R^j x \in  [z-\fp{q_k\alpha},z).$$
\item[{\rm (b)}] We have $$|\{0\leq j \leq q_{k+1}+q_k:\exists x
\in[-\|q_k\alpha\|,0) \text{ with }R^j x \in
[z-\|q_k\alpha\|,z)\}|\leq 4.$$
\item[{\rm (c)}] There exists $j \in
\{1,...,q_{k+1}+q_k\}$ such that 
\begin{displaymath}
\lambda(R^j([-\|{q_{k+1}\alpha}\|,0))\cap [z-\fp{q_k\alpha},z)) >
\frac{1}{4}\|q_{k+1} \alpha\|
\end{displaymath}
\end{itemize}
\end{claim}

\begin{proof}[Proof of Claim~\ref{claim:hit:both}] Since $R$ is
minimal, if $J'$ is an interval then for every $x$, $R^ix\in J'$ for
some $0\leq i< \underset{x \in J'}{\max}\min\{j>0:R^jx \in
J'\}$. For any interval of size $\|q_k\alpha\|$ this is
$q_{k+1}+q_k$, see Lemma~\ref{lemma:ret:time}. This proves (a).  By
Lemma~\ref{lemma:orbit:sep}, for any $\ell$ there exists at most two
$j$ in $\{\ell,...,\ell+q_{k+1}-1\}$ so that there exists
$x \in[-\|q_k\alpha\|,0) \text{ with }R^j x \in
[z-\|q_k\alpha\|,z)$. Since $q_{k+1}+q_k<2q_{k+1}$ this implies
(b). The statement (c) follows from (a) and (b). Indeed, let
$\Delta$ denote the set of $j$ in part (b). Then, in view of (b),
$|\Delta| \le 4$. For each $j \in \Delta$, let $I_j$ denote the set
of $x \in [-\|q_{k+1} \alpha\|, 0)$ such that
$R^j x \in [z-\fp{q_k\alpha},z)$. Then, by (a),
$[-\|q_{k+1} \alpha\|,0) = \bigcup_{j \in \Delta} I_j$. Thus, there
exists $j \in \Delta$ such that $\lambda(I_j) \ge \frac 1 4 \|q_{k+1} \alpha\|$.
\end{proof}

We now continue the proof of Lemma~\ref{lemma:first:sum}. 
We choose $j$, $0 \le j \le q_{k+1} + q_k$ so that 
\begin{displaymath}
\lambda(R^j([-\|q_{k}\alpha\|,0))\cap [z-\|q_k\alpha\|,z)) \text{ is maximal.}
\end{displaymath}
Let 
\begin{displaymath}
I=R^{-j}([z-\|q_k\alpha\|,z)) \cap [-\|q_{k+1}\alpha\|,0)). 
\end{displaymath}
Note that by Claim~\ref{claim:hit:both}(c),  
$\lambda(I)>\frac{1}{4} \|q_{k+1}\alpha\|$.

\begin{claim}
\label{claim:choose:E0}
Either
\begin{equation}
\label{eq:j:small}
(j - c_2 q_k, j+q_k) \subset \{1, q_{k+1}+q_k - 1\}
\end{equation}
or
\begin{equation}
\label{eq:j:big}
(j - q_k, j+ c_2 q_k) \subset \{1, q_{k+1}+q_k - 1\}
\end{equation}
\end{claim}

\begin{proof}[Proof of Claim~\ref{claim:choose:E0}]
Note that by (A4), for $0<i<q_k/C_4$,
\begin{equation}
\label{eq:first:comment}
\text{ if $x \in [-\|q_k\alpha\|,0)$ then $R^i x \not\in [z-\|q_k\alpha\|,z)$.}
\end{equation}
Therefore, $j > q_k/C_4 \ge c_2 q_k$. 
Similarly, for $0<i<q_k/C_4$,
\begin{equation}
\label{eq:first:comment:prime}
\text{if 
$x \in [z-\|q_k\alpha\|,z)$ then $R^i x \not\in [-\|q_k\alpha\|,0)$.}
\end{equation}
Therefore, $j < q_{k+1} + q_k - \frac{q_k}{C_4} \le q_{k+1} + q_k -
c_2 q_k$. 

Now if $j < q_{k+1}$ then (\ref{eq:j:small}) holds, and if $j > q_k$
then (\ref{eq:j:big}) holds. This completes the proof of Claim~\ref{claim:choose:E0}. 
\end{proof}

\begin{claim}
\label{claim:E0:works}
Suppose (\ref{eq:j:small}) holds and $\ell \in (j - c_2 q_k, j+q_k)$
or (\ref{eq:j:big}) holds and $\ell \in (j-q_k, j+c_2 q_k)$. Also
assume that $\ell \ne j$. Then, 
\begin{displaymath}
R^\ell I \cap (([-\|q_k\alpha\|,0) \cup [z-\|q_k\alpha\|,z)) = \emptyset.
\end{displaymath}
\end{claim}

\begin{proof}[Proof of Claim~\ref{claim:E0:works}]
Recall that $I \subset [-\|q_{k+1} \alpha \|,0)$ and thus by
Lemma~\ref{lemma:ret:time}, 
\begin{displaymath}
R^\ell I \cap [-\|q_k\alpha\|,0) = \emptyset \qquad \text{ for $1 \le
  \ell \le q_{k+1} + q_k - 1$.}
\end{displaymath}
Also, by Lemma~\ref{lemma:ret:time}, the return time of any point in
the interval $[z-\|q_k\alpha\|,z)$ to itself is at least $q_{k+1} >
q_k$. Thus, for $\ell$ such that $|\ell - j| < q_k$, 
\begin{displaymath}
R^\ell I \cap [z-\|q_k\alpha\|,z) = \emptyset.
\end{displaymath}
Claim~\ref{claim:E0:works} follows. 
\end{proof}

We now continue the proof of Lemma~\ref{lemma:first:sum}. Let $r =
\min(c_2 q_k, (m-m') w_k)$. Recall that $(m-m')w_k < q_k$. 
If (\ref{eq:j:small}) holds, let
\begin{displaymath}
E = (j-r, j-r+c_2 q_k), \qquad \text{ so that $E+[0,(m-m')w_k) \subset (j-c_2 q_k, j+q_k)$.}
\end{displaymath}
If (\ref{eq:j:big}) holds, let
\begin{multline*}
E = (j-(m-m')w_k, j-(m-m')w_k + c_2 q_k) \\ \text{ so that
  $E+[0,(m-m')w_k) \subset (j- q_k, j+c_2 q_k)$.}
\end{multline*}
Then, for all $i \in E$, 
\begin{displaymath}
i \le j \le i+(m-m')w_k.
\end{displaymath}
Hence, by Claim~\ref{claim:E0:works}, for all $x \in I$ and for all $i
\in E$, 
\begin{equation}
\label{eq:tmp:Rellix}
|\{R^{i+\ell} x\}_{\ell=1}^{(m-m')w_k}\cap ([-\|q_k\alpha\|,0)\cup
[z-\|q_k\alpha\|,z))|=1.
\end{equation}
(the only contribution is from the case where $i + \ell = j$). 
Therefore, in view of (\ref{eq:sum:rearrangement}) and
(\ref{eq:difference:non:zero}), for $x \in R^{-m' w_k} I$ and $\ell \in E$, 
(\ref{eq:E:property}) holds. 

From the definition, $|E| \ge c_2 q_k$. We now estimate $\lambda(I)$. 
By Lemma~\ref{lemma:q:k:plus:1:ak:qk}, 
\begin{displaymath}
q_{k+2} = a_{k+2} q_{k+1} + q_{k} < (a_{k+2}+1) q_{k+1}.
\end{displaymath}
\begin{displaymath}
q_{k+1} = a_{k+1} q_k + q_{k-1} < (a_{k+1} + 1) q_{k}
\end{displaymath}
By Lemma~\ref{lemma:good:bound},
\begin{multline*}
\lambda(I) \ge \frac{1}{4} \|q_{k+1} \alpha\| \ge \frac{1}{4}
\frac{1}{q_{k+2} + q_{k+1}} \ge \frac{1}{4} \frac{1}{(a_{k+2} +
  2)q_{k+1}} \ge \\ \ge \frac{1}{4} \frac{1}{(a_{k+2} + 2)(a_{k+1}+1) q_k}
\ge \frac{1}{4} \frac{a_{k+1}}{(a_{k+2} + 2)(a_{k+1}+1)}\|q_k \alpha\|.
\end{multline*}
Thus, by (A3),
\begin{displaymath}
\lambda(I) \ge \frac{1}{8(C_3+2)} \|q_k \alpha \|.
\end{displaymath}
This completes the proof of Lemma~\ref{lemma:first:sum}. 
\end{proof}


\subsection{Proof of Proposition \ref{prop:technical:cond3} from
  Lemma~\ref{lemma:first:sum}}
\label{sec:subsec:proof:of:prop:technical}

Recall $w_k =\lfloor \frac{r_k}{\lambda(J)} \rfloor$ as above. 

\begin{cor}(Corollary to Lemma \ref{lemma:first:sum}) Given $w_k$ so
  that $q_k <(m-m')w_k<q_{k+1}$ as before there exists $\hat{m}\in
  \{m,m'\}$ and a set $A_k$ with $\lambda(A_k) \geq \tilde{c}$
  (depending on our non-divergence condition, that is
  $C_1,...,C_4,C'$) so that for all $x \in A_k$ we
  have 
\begin{equation}
\label{eq:cor:first:sum}
\sum_{i=0}^{\hat{m}w_k-1}\chi_J(R^ix)-\sum_{i=0}^{\hat{m}w_k-1}\chi_J(R^{i+q_k}x)=d\in \{-3,-2,-1,1,2,3\}.
\end{equation}
\end{cor}
\begin{proof}
Lemma \ref{lemma:first:sum} establishes that
there exists $\bar{c} > 0$ and $\bar{c}_1 > 0$, 
and for an infinite sequence of $k  \in \natls$ there exists an interval 
$I' \subset [0,1]$ with $\lambda(I')>\bar{c}\|q_k \alpha\|$ so that
for any $x\in I'$ there exists $H_x \subset \{0,1, \dots ,q_k-1\}$ 
with $|H_x|>\bar{c}_1 q_k$ so that  for any $x \in I'$ and any $\ell \in
H_x$, and any $w_k$ with $(m-m')w_k \in [q_{k-1},q_k)$ we have
\begin{multline}
\label{eq:the:main:sum}
\sum_{i=0}^{mw_k-1}(\chi_J(R^i R^\ell x)-\chi_J(R^{i+q_k}R^\ell
x)) \\ - \sum_{i=0}^{m'w_k-1}(\chi_J(R^i R^\ell
x)-\chi_J(R^{i+q_k}R^\ell x)) \in \{-1,1\}.
\end{multline}

Also note that by Lemma \ref{lemma:DK}, $\psi_{q_k}$ takes at most 5
values (which are also consecutive) and so for any $s \in \natls$ and
any $x$ we have
\begin{equation}
\label{eq:rearrange:psi}
\psi_s(x)-\psi_s(R^{q_k}x)=\psi_{q_k}(x)-\psi_{q_k}(R^s x)\in \{-4,\dots, 4\}. 
\end{equation}
Note that the left-hand-side of (\ref{eq:the:main:sum}) is
\begin{displaymath}
(\psi_{mw_k}(R^\ell x) - \psi_{m w_k}(R^{q_k} R^\ell x)) - 
(\psi_{m'w_k}(R^\ell x) - \psi_{m' w_k}(R^{q_k} R^\ell x)) \equiv
S_1(R^\ell x) - S_2(R^\ell x).  
\end{displaymath}
By (\ref{eq:the:main:sum}), for all $x \in I'$ and for all $\ell \in
H_x$, $S_1(R^\ell x) - S_2(R^\ell x) 
\in \{-1,1\}$, and by (\ref{eq:rearrange:psi}),
we have $|S_1(R^\ell x)| \le 4$, and $|S_2(R^\ell x)| \le 4$. It follows
that for all $x \in I'$ and all $\ell \in H_x$, 
\begin{displaymath}
S_i(R^\ell x) \in \{-3,-2,-1,1,2,3\} \qquad\text{ for some $i \in \{1,2\}$.}
\end{displaymath}
Thus, there exists $\hat{m}\in \{m,m'\}$ and $d \in \{-3,-2,-1,1,2,3\}$ and 
a set $A_k$ with
\begin{displaymath}
\lambda(A_k) \ge \frac{1}{12} |H_x| \lambda(I') =  \frac{\bar{c}\bar{c}_1 q_k
  \|q_k\alpha\|}{12},
\end{displaymath}
so that for for $x \in A_k$ (\ref{eq:cor:first:sum}) holds. 
\end{proof}
We frequently use the following trivial result in this section.

\begin{lemma}\label{lem:size bound} If $\lambda(B)\geq \gamma$ and $B$ is the union of at most $\ell$ intervals then there exists $B'\subset B$ with $\lambda(B')\geq\frac 1 2 \lambda(B)$ and $B'$ is the union of intervals of size at least $\frac{\gamma}{2\ell}$. 
\end{lemma}


\begin{lemma} There exists $A_k'\subset A_k$ with $\lambda(A_k')>\frac{1}{2} \lambda(A_k)$ and so that $A_k'$ is made of at most $4q_k+1$ intervals with length at least $\frac {\tilde{c}}2 \frac 1 {4q_k+1}$.
\end{lemma}
\begin{proof}  Recall that $A_k$ is a level set of $$\sum_{i=0}^{\hat{m}w_i-1}\chi_J(R^ix)-\sum_{i=0}^{\hat{m}w_i-1}\chi_J(R^iR^{q_k}x)=\sum_{i=0}^{q_k-1}\chi_J(R^ix)-\sum_{i=\hat{m}w_i}^{\hat{m}w_i +q_k -1}\chi_J(R^ix),$$ a function which has at most $4q_k$ discontinuities.
The lemma follows from Lemma \ref{lem:size bound} since this implies that any level set is made of at most $4q_k+1$ intervals. 
\end{proof}


In the previous results we have proved properties of a level set of $\sum_{i=0}^{\hat{m}w_k-1}\chi_J(R^ix)-\sum_{i=0}^{\hat{m}w_k-1}\chi_J(R^{i+q_k}x)$. In this lemma we relate that to proving nice properties about a set, $G_\ell=\{x:S^{\hat{m}n}x=R^\ell x\}$ for some $\ell$,  to obtain the set $\tilde{A}$ in Proposition \ref{prop:technical:cond3}.

\begin{lemma} 
\label{lemma:Ak:twoprime} 
For all large enough $k$ there exists $\hat{m} \in \{m,m'\}, \, d\in \{-3,-2,-1,1,2,3\}$ and a set $\tilde{A}_k$ with $\lambda(\tilde{A}_k)> \frac{\tilde{c}}4$ and which is the union of at most $8q_k$ intervals  of size at
  least $\frac {\tilde{c}}2 \frac 1 {8\cdot 4q_k}$ so that for $x \in \tilde{A}_k$, 
\begin{equation}
\label{eq:Ak:twoprime}
R^{q_k}S^{\hat{m}r_k}(x)=S^d S^{\hat{m}r_k}(R^{q_k}x). 
\end{equation}
\end{lemma}

\begin{proof}
By Lemma~\ref{lemma:DK}, for all $h,j \in
\natls$, and any $x \in [0,1]$, 
\begin{equation}
\label{eq:iterated:djk}
\left|-h q_j \lambda(J)+ \sum_{i=0}^{h q_j-1} \chi_J(R^i x)\right|  \le 2h.
\end{equation}
Let $0 < N < q_b$ be a positive integer, and write 
\begin{displaymath}
N = \sum_{i=0}^{b-1} h_i q_i, \qquad\text{ where $h_i \in \zed$, $0 \le h_i \le a_{i+1}$ and $h_{b-1}<\frac{N}{q_{b-1}}$.}
\end{displaymath}
Let $D_0 = 0$, and for $0 < j \le b$, let 
\begin{displaymath}
D_j = \sum_{i=0}^{j} h_i q_i,  \qquad\text{where $q_0 = 1$.}
\end{displaymath}
Note that $D_b = N$. Then, by (\ref{eq:iterated:djk}), 
for $x \in [0,1]$, 
\begin{align} \label{eq:bounding}
\left| - N \lambda(J)+\sum_{i=0}^{N-1}\chi_J(R^i x)\right| & = \notag
\left| \sum_{j=0}^{b-1} \left( -(D_{j+1}-D_j) \lambda(J) + \sum_{i =
      D_j}^{D_{j+1}-1} \chi_J(R^i x) \right) \right| \notag \\
& = \left| \sum_{j=0}^{b-1} \left( -h_j q_j \lambda(J) + \sum_{i
      = 0}^{h_j q_j-1} \chi_J(R^i R^{D_j} x) \right) \right| \notag \\
& \le 2 \sum_{j=0}^{b-2}  a_{j+1}+\frac{N}{q_{b-1}} \qquad\text{ by (\ref{eq:iterated:djk})} \notag \\
&=o(N) + \frac{N}{q_{b-1}}=o(N).  
\end{align}
For $x \in J$ define $N(x)$ so that
\begin{displaymath}
\hat{m} r_k = \sum_{i=0}^{N(x)} \chi_J(R^i x).
\end{displaymath}
We now apply (\ref{eq:bounding}) with $N(x)$ instead of $N$. We
obtain that for each
$x \in J$ there exists $N(x) \in \natls$ so that
\begin{equation}
\label{eq:sum:is:hatm:rk}
\hat{m}r_k = \sum_{i=0}^{N(x)} \chi_J(R^i x), \qquad\text{ and  $|N(x)
  - \hat{m} w_k| \le 4\sum_{i=0}^{b-2} a_i+\frac{\hat{m}w_k}{q_{b-1}}$.}
\end{equation}
Since 
$$
\hat{m}w_k<q_{k+1}\frac{\max\{m,m'\}}{m-m'}<(C_2+1)q_{k}\frac{\max\{m,m'\}}{m-m'}$$ we have that there exists $D$ so that for all $k$, $\hat{m}w_k<Dq_k$. 
Also by Lemma \ref{lem:sum bound}
 $$
\sum_{j=0}^{b-2}  a_{j+1} +\frac{N(x)}{q_{b-1}} = o(q_{b-1})+o(N(x))=o(N(x))$$ and so $N(x)=\hat{m}w_k+o(\hat{m}w_k)$.
 Therefore,
 for all large enough $k$ we
have $|N(x) - \hat{m} w_k|<\frac{\tilde{c}}{32}q_k$. 

Observe that if  
\begin{equation}
\label{eq:Nx:different:from:mwk}
\sum_{i=0}^{N(x)}\chi_J(R^ix)-\sum_{i=0}^{N(x)}\chi_J(R^iR^{q_k}x)\neq
\sum_{i=0}^{\hat{m} w_k}\chi_J(R^ix)-\sum_{i=0}^{\hat{m}w_k} 
\chi_J(R^iR^{q_k}x),
\end{equation}
then $R^jx $ is in one of two intervals of size at most
$\|q_k\alpha\|$ for some
\begin{displaymath}
j \in [\min(N(x), \hat{m} w_k), \max(N(x), \hat{m} w_k)]. 
\end{displaymath}
It follows that there exists $\tilde{A}_k \subset A_k'$ with $\lambda(\tilde{A}_k) >
\frac{\tilde{c}}4$ which is a union of intervals each of size at
least $\frac {\tilde{c}}2 \frac 1 {8\cdot 4q_k}$, such that for $x
\in \tilde{A}_k$ (\ref{eq:Nx:different:from:mwk}) does not hold. 
Indeed, we are removing at most $2\frac{\tilde{c}}{32}q_k$ intervals
of size $\|q_k\alpha\|$ so we obtain a set of measure at least
$\lambda(A_k')-\frac{\tilde{c}}{16}$ that is a union of at most $4q_k+1+4\frac{\tilde{c}}{32}q_k$ intervals and can invoke Lemma~\ref{lem:size bound}.

 Suppose
$x \in \tilde{A}_k \subset A_k$. Then, in view of (\ref{eq:cor:first:sum}), 
\begin{displaymath}
\sum_{i=0}^{N(x)}\chi_J(R^ix)-\sum_{i=0}^{N(x)}\chi_J(R^iR^{q_k}x) = -d, 
\end{displaymath}
where $d \in \{-3,-2,-1,1,2,3\}$. 
In view of (\ref{eq:sum:is:hatm:rk}), this can be rewritten as
\begin{equation}
\label{eq:sum:is:d:plus:hatm:rk}
d+\hat{m} r_k= \sum_{i=0}^{N(x)}\chi_J(R^iR^{q_k}x).
\end{equation}
Now, in view of (\ref{eq:time:change:general}),
(\ref{eq:sum:is:hatm:rk}) and (\ref{eq:sum:is:d:plus:hatm:rk}),
\begin{displaymath}
S^{\hat{m} r_k} x = R^{N(x)} x \qquad\text{and}\qquad 
S^{d+\hat{m} r_k} R^{q_k} x = R^{N(x)}
R^{q_k} x. 
\end{displaymath}
The equation (\ref{eq:Ak:twoprime}) follows. 
\end{proof}


To complete the proof we need to show $\{S^{i\hat{m}}x\}$ hits
$\tilde{A}_k$ frequently enough. Lemma \ref{lem:hit interval} lets us
that show $R$ orbits hit $\tilde{A}_k$ frequently enough. The key observation we use is that if we define $j_i$ by $S^{\hat{m}i}x=R^{j_i}x$ then $j_{i+1}-j_i\leq 2 \hat{m}$. We call a set with this property \emph{$2\hat{m}$ dense}. This motivates us to build an auxiliary set, $\hat{A}_k$, so that the hits of an $R$ orbit to $\hat{A}_k$ give a lower bound for the hits of an $S^{\hat{m}}$ orbit to $\tilde{A}_k$.
\begin{proof}[Proof of Proposition \ref{prop:technical:cond3}]
We assume $k$ is large enough so that Lemma \ref{lemma:Ak:twoprime}  holds, $q_k>16 \hat{m}$ and $4\hat{m}\|q_k\alpha\|<\frac{\tilde{c}}{16}$.

Consider $\tilde{A}_k$ and remove from it all $x$ so that 
$$\sum_{i=0}^{\hat{m}w_i}\chi_J(R^i(x))-\sum_{i=0}^{\hat{m}w_i}\chi_J(R^{i+q_k}x)\neq \sum_{i=0}^{\hat{m}w_i}\chi_J(R^iR^j(x))-\sum_{i=0}^{\hat{m}w_i}\chi_J(R^{i+q_k}R^jx)$$ for some $j<2\hat{m}$.
This means that if $x$ remains in this set then
\begin{equation}
\label{eq:stretch}
\exists\ j\leq 0\leq k \text{ so that } k-j\geq 2\hat{m}
\text{ and }R^\ell x\in \tilde{A}_k \text{ for all } \ell \in \{j,\dots, k\}.
\end{equation}
The set remaining, $A'''_k$ has measure at least
$\lambda(\tilde{A}_k)-4\hat{m}\|q_k\alpha\|$ and is made up of at most $8\hat{m}+8q_k$ intervals. (Indeed, we are removing at most $2\hat{m}$ pre-images of 2 intervals of size $\|q_k\alpha\|$.) 
 Let $\hat{A}_k$ be a subset of $A'''_k$ of measure at least $\frac 1
 2 \lambda(A'''_k)$ made of intervals of length at least
 $\frac{\tilde{c}}{32 \cdot 64(q_{k})}$. (Indeed, we invoke Lemma \ref{lem:size bound} using that $A'''_k$ is a set of measure at least $\frac{\tilde{c}}8$ which is made up of at most $16q_k$ intervals.) 
 By the estimate of the size of intervals in $\hat{A}_k$, 
Lemma \ref{lem:hit interval} implies that if $\frac{q_{k+u}}{q_k}$ is
large enough we have for $t > q_{k+u}$,
\begin{equation}
\label{eq:sep:hit:bound}
\frac{1}{t}\sum_{i=0}^{t} \chi_{\hat{A}_k}(R^ix)>\frac 1 {4}
 \lambda(\hat{A}_k). 
\end{equation}

Because $R^ix \in \hat{A}_k$, the equation (\ref{eq:stretch})
implies there exists $j\leq i\leq k$ with $k-j\geq 2\hat{m}$ so that
$R^\ell x\in \tilde{A}_k$ for all $\ell \in \{j,\dots,k\}$. Then,  we have 
\begin{displaymath}
\frac 1 {2\hat{m}}\sum_{i=0}^t \chi_{\hat{A}_k}(R^ix)\leq |\{i \in \mathcal{C}:R^ix\in \tilde{A}_k\}|, 
\end{displaymath}
where $\mathcal{C}$ is any $2\hat{m}$ dense subset of
$\{0,...,t+2\hat{m}\}$. 

Observing that $\{j\in [0,k+2m]: \exists i \text{ with }S^{\hat{m}i}x=R^jx\}$
is $2\hat{m}$ dense this implies that 
\begin{displaymath}
\frac 1 {2\hat{m}}\sum_{i=0}^t \chi_{\hat{A}_k}(R^ix)\leq \sum_{i=0}^{N_t(x)} \chi_{\tilde{A}_k}(S^{i\hat{m}}x)
\end{displaymath}
where $N_t(x)=\min\{j:S^{\hat{m}j}x=R^\ell x \text{ with }\ell\geq t\}$.
We obtain the proposition with $\hat{C}=16$.  Indeed,
$\lambda(\hat{A})\geq \frac 1 4 \lambda(\tilde{A})$ and so by
(\ref{eq:sep:hit:bound}) we have that for $t>q_{k+u}$, 
we have that
\begin{equation}
\label{eq:new:sep:hit:bound}
\frac{1}{t}\sum_{i=0}^t
\chi_{\hat{A}_k}(R^ix)>\frac{\lambda(\tilde{A})}{16}.
\end{equation}
Lastly,
$G_x:=\{j:\exists i \text{ with }S^{\hat{m}i}x=R^jx\}$ is at least
$\hat{m}$ separated (that is if $j \in G_x$ and $|i-j|<\hat{m}$ then
$i \notin G_x$) and so $N_t(x)\leq \frac{t}{\hat{m}}$, letting us obtain,
using (\ref{eq:new:sep:hit:bound}),
\begin{displaymath}
\frac{1}{2\hat{m} \cdot 16}\lambda(\tilde{A})\leq \frac{1}{2\hat{m}t}\sum_{i=0}^t \chi_{\hat{A}_k}(R^ix)\leq \frac 1 {t}\sum_{i=0}^{N_t(x)} \chi_{\tilde{A}_k}(S^{i\hat{m}}x)\leq \frac {1} {\hat{m}N_t(x)}\sum_{i=0}^{N_t(x)} \chi_{\tilde{A}_k}(S^{i\hat{m}}x).
\end{displaymath}
Multiplying the sequence of inequalities by $\hat{m}$ completes the estimate.
\end{proof}

\section{Renormalization}\label{sec:renorm}
Recall that $X$ is a torus with two marked points related to a 3-IET, $T$ and $\hat{X}$ is the torus obtained by forgetting the two marked points.

\bold{Divergence in the space of tori, $\mathcal{M}_1$:} By Mahler's compactness criterion the divergence of $g_t\hat{X}$ is controlled by the shortest (non-homotopicaly trivial) simple closed curve 
on $g_t\hat{X}$. This sequence is given by curves $\gamma_k$ with vertical holonomy $q_k$ and horizontal holonomy $\pm |q_k\alpha-p_k|=\pm \fp{ q_k\alpha}$. Coarsely, this curve is contracted from $t=0$ to   $t= \log(q_k\sqrt{a_{k+1}})$ and then expanded. Additionally, there is a fixed compact set $\hat{K}$ so that $g_{\log(q_k)}\hat{X} \in \hat{K}$ for all $k$ (and in particular $|\gamma_k|$ is proportional to 1 at $g_{\log(q_k)}$). 
\begin{lemma}\label{lem:rot divergence} For any $\hat{K} \subset \mathcal{M}_1$ there exists $\ell_{\hat{K}}$ so that for all $k$ we have $|\{t\in [\log(q_k),\log(q_{k+1})):g_t\hat{X} \in \hat{K}\}|<\ell_{\hat{K}}$. 
\end{lemma}
\begin{proof}For any $\hat{K}$ there exists $\delta$ so that if $
  g_s\hat{X} \in \hat{K}$ then the shortest simple closed curve on
  $g_s\hat{X}$ is at least $ \delta$. As in the previous paragraph,
  consider the curve $\gamma_k$ on $\hat{X}$, with vertical holonomy $q_k$ and
  horizontal holonomy $\pm |q_k \alpha - p_k|$. On $g_s \hat{X}$ the curve
  $g_s \gamma_k$ has vertical holonomy $e^{-s} q_k$ and horizontal
  holonomy $\pm e^s |q_k \alpha - p_k|$.   
If $s\in [\log(q_k),\log(q_{k+1})]$ then, since we are assuming that
the length of $g_s \gamma_k$ is at least $\delta$, we must have $e^s |q_k\alpha-p_k|\geq \frac \delta 2 $ or $e^{-s}q_k\geq \frac \delta 2 $. By Lemma \ref{lemma:good:bound} the first condition can only hold if $e^s>\frac{\delta}2 {a_{k+1}q_k}$. Noticing that $a_{k+1}q_k>\frac 1 2 q_{k+1}$ this implies $s>\log(q_{k+1}) +2\log(2)+\log(\delta)$. The second condition can only hold if $s<2q_k \frac1{\delta}$. The lemma follows with $\ell_{\hat{K}}=-2\log(\delta)-3\log(2).$
\end{proof}

We now assume the assumption of Proposition \ref{prop:good:assump}. 
This means there exists a compact set $\mathcal{K}\subset \mathcal{M}_{1,2}$ so that $\underset{T \to \infty}{\limsup} \, \frac 1 T |\{0<t<T:g_tX \in \mathcal{K}\}|=c>0$. Let $D_1,...$ be a sequence chosen so that $$\frac 1 {D_i} |\{0<t<D_i:g_tX\in \mathcal{K}\}|>\frac {99c}{100}$$ and
 $$\underset{\zeta>D_1}{\sup}\, |\{0<t<\zeta: g_tX\in \mathcal{K}\}|-c<\frac c {100}.$$ 

 The next lemma is used to obtain (A7).
 \begin{lemma}\label{lem:cont in cpct} For all $r>0$ for all $i$ large enough we have 
$$|\{t<D_i:g_tX \in \mathcal{K}, |\{s\in [t,t+r]:g_sX \in \mathcal{K}\}|>\frac c {99} r\}|>\frac {8c} {9}.$$ 
\end{lemma}
\begin{proof}
This is a standard application of the Vitali covering lemma. Indeed let 
$$B=\{t<D_i:g_tX \in \mathcal{K}, \lambda(\{s\in [t,t+r]:g_sX \in \mathcal{K}\})>\frac c {99} r\}$$ and so for each $t \in B$ we have $\lambda(\{s\in [t,t+r]:g_sX \in \mathcal{K}\})<\frac c {99}r.$
 By applying the Vitali covering lemma to the intervals $[t,t+r]$ where $t\in B$, we may take a disjoint subcollection of these intervals $I_1,\dots I_\ell$ so that 
\begin{equation}\label{eq:vitali}\lambda(\cup_{i=1}^\ell \{s\in I_i:g_s{X}\in \mathcal{K})>\frac 1 3 \lambda( \cup_{t \in B}\{s\in[t,t+r]:g_sX\in \mathcal{K})\geq \frac{\lambda(B)}3.
\end{equation}
Indeed let $U_1=\{[t,t+r]:t\in B\}$ and choose $I_1$ to be an interval $[t,t+r]$ in this set so that $\lambda(\{s\in [t,t+r]:g_sX\in \mathcal{K}\}$ is maximal. Let $U_2=\{[t,t+r]:t \in B \text{ and }[t,t+r]\cap I_1=\emptyset$ and let $I_2$  be an interval $[t,t+r]$ in this set so that $\lambda(\{s\in [t,t+r]:g_sX\in \mathcal{K}\}$ is maximal. Also observe 
$$\lambda(\{s: s\in [\tau,\tau+r] \text{ with }\tau \in B, \, [\tau,\tau+r]\cap I_1\neq \emptyset \text{ and } g_sX\in \mathcal{K}\})\leq 2\lambda(\{s\in I_1:g_sX\in \mathcal{K}\}).$$ Repeating this procedure we obtain our intervals $I_1,\dots, I_\ell$.

Having established (\ref{eq:vitali}) we see at most $\frac c {99}r$ of the points in each interval are in $\mathcal{K}$ and the measure of the union of these intervals is at most $D_i$. This is a contradiction unless $\lambda(B)\leq \frac c {33}D_i+r$. 
\end{proof}
By the same proof we obtain:
\begin{lemma}\label{lem:vit}For all $\epsilon,\gamma>0$ there exists $\delta>0$ so that if $|A\cap[0,R]|>\gamma R$ then 
$$|\{t\in [0,R]\cap A:\lambda(\{t+s\in A\}_{t\in [0,T]})>\delta \gamma\}|>(1-\epsilon)\gamma R-2T.$$
\end{lemma}

Let $f(t)=\max\{j: q_j\leq e^{t}\}$. 
The next lemma is used to obtain (A8). 
\begin{lemma}\label{lem:in cpct} For all $r,\epsilon>0$ there exists $\hat{K}\subset \mathcal{M}_{1}$ so that for all $i$ large enough we have 
$$|\{t<D_i:g_tX\in \mathcal{K}, g_\ell \hat{X} \in \hat{K} \text{ for all }\ell\in [\log(q_{f(t+r)}),\log(q_{f(t+r)+1})]\}|\geq (1-\epsilon) \frac {99c} {100} D_i.$$
\end{lemma} 
We use the following straightforward consequence of Lemma \ref{lem:vit}. 
\\
\noindent
\textbf{Sublemma:} If $h:[0,\infty)\to \{0,1\}$ has $ \frac 1 R \int_0^R h(t)\geq \frac{99c}{100}$ then for all $\epsilon'>0,r, \ell$ there exists $L$ so that 
\begin{multline}
 |\{t<R:h(t)=1 \text { and there exists  }0<\rho\leq r \text{ so that }\\ h(t+s)=0 \text{ for all but a set of measure $\ell$ of }
s \in [t+\rho,t+\rho+L] \} |<\epsilon' cR+2L.
\end{multline}
\begin{proof}[Proof of Sublemma] Apply Lemma \ref{lem:vit} with $\epsilon=\epsilon'$ and $\gamma =\frac {99c}{100}$ to obtain $\delta$. Choose $L$ so that $\delta L>\ell+r$.
\end{proof}
\begin{proof}[Proof of Lemma \ref{lem:in cpct}]  Let $\hat{C}$ be the compact set in $\mathcal{M}$ given by projecting $\mathcal{K}$ to $\mathcal{M}$ by forgetting the marked points. Let $\ell=\ell_{\hat{C}}$ as in Lemma \ref{lem:rot divergence} and $\delta$ be the shortest simple closed curve on any surface in $\hat{C}$. Obtain $L$ from the sublemma with $r=r$, $\epsilon'=\frac{\epsilon}2$, $\ell=\ell$. Let $\hat{\mathcal{K}}$ denote the set of all tori whose shortest simple closed curve is at least $\delta e^{-L}$. The lemma holds for this $\hat{\mathcal{K}}$.  Indeed, if  $g_s\hat{X} \notin \hat{\mathcal{K}}$ then by examining the size of the shortest simple closed curve we see $g_{s+\tau}X\notin \mathcal{K}$ for all $-L<\tau<L$.  That is, considering $A=\{s:g_sX \in \mathcal{K}\}$ and $\rho=\log(q_{f(t+r)})-t$  we are asking that $|\{s\in [t+\rho,t+\rho+L]:s\in A\}|<\ell$. So by the sublemma the set of such $t$ 
 has small density and so we have the lemma.
\end{proof}

We now begin the derivation of (A5), (A6) and (A8). 
\begin{lemma}
\label{lemma:the:time}
For all $t$  there exists $-2\leq s\leq 2$ so that either
\begin{itemize}
\item there exists $k$ with $a_{k+1}>4$ and $i\leq \lfloor \frac{a_{k+1}}4 \rfloor$ so that $e^{t+s}=iq_k$
\item or there exists $k$ so that $e^{t+s}=q_k$. 
\end{itemize}
\end{lemma}

\begin{proof} Let $j=f(t)$. 
If $a_{j+1}\leq4$ then since $e^2>5$ we may choose $s$ so that $e^{t+s}=q_j$ (and so $k=j$). If  $a_{j+1}>4$ and $i>\frac {q_{j+1}}2$  choose $s$ so that $e^{t+s}=q_{j+1}$ (and so $k=j+1$). 
Otherwise choose $s$ so that  $e^{s+t}=iq_j$ with $i\leq \lfloor \frac{a_{j+1}}4\rfloor$ (and so $k=j$). 
\end{proof}

The following is a corollary of Lemma~\ref{lemma:the:time} and
Lemma~\ref{lemma:happy:times}:
\begin{cor}
\label{cor:happy times}
For all $\eta>0$ there exists $\rho>0$ so that for all $t$, there exists $-2\leq s\leq 2$ with
 $\lambda(\psi_{e^{t+s}}(j))> \rho$ for some $j$ so that  $$\begin{cases} \text{either } 
j-\min\{\ell: \psi_{e^{s+t}}^{-1}(\ell) \neq \emptyset\}\leq 2  \text{ and } \lambda(\psi_{e^{s+t}}^{-1}((0,j))<\rho \eta\\ \text{or }
\max\{\ell:\psi_{e^{s+t}}^{-1}(\ell)\neq \emptyset\}-j\leq 2   \text{ and } \lambda(\psi_{e^{s+t}}^{-1}(j,\infty))<\rho\eta.
\end{cases}$$  
\end{cor}
For the proof we use the following trivial result:\\

\noindent
\textbf{Sublemma:}
For all $\eta>0$ there exists $\rho>0$ so that if $x_0,x_1,x_2 >0$ and $\sum x_i>\frac 1 {12}$ then there exists $i$ so that $x_i>\rho$ and $\sum_{j=0}^{i-1}x_j<\eta \rho$. Also there exists $\ell$ so that $x_\ell>\rho$ and $\sum_{i=\ell+1}^2 x_i<\eta \rho$.

\noindent 
\begin{proof}[Proof of Corollary \ref{cor:happy times}] Applying Lemma \ref{lemma:the:time} we obtain $e^{t+s}$. If $q_k=e^{t+s}$ then by Lemma \ref{lemma:DK} we have $\psi_{q_k}$ takes at most 5 values that are consecutive.  Letting $x_i$ be the measure of the $i^{th}$ level set and appying the sublemma
implies the corollary. Otherwise by Lemma \ref{lemma:happy:times} and the sublemma imply the corollary.
\end{proof}

The next lemma is used to obtain (A0)-(A4). Its proof is similar to Lemma \ref{lem:in cpct} and is omitted. 
\begin{lemma}
\label{lem:A0 to A4} 
Given any $\epsilon>0$ there exists $M$ so that $\lambda(E_2)=\lambda(\{t<D_i:g_tX\in \mathcal{K} \text{ and for all } s\in [-3,3] \text{ we have }a_{f(t+s)},a_{f(t+s)+1},a_{f(t+s)+2}<M \})>D_i(1-\epsilon)\frac{99c}{100}$.  
\end{lemma}

 By choosing $\epsilon=\frac 1 9$ in Lemmas \ref{lem:in cpct}, \ref{lem:cont in cpct} and \ref{lem:A0 to A4}, for each $r$, we may choose a sequence of $t$ going to infinity which is simultaneously in the three sets whose measure is bounded from below in  these Lemmas. For each $t$ there exists $s$ as in Lemma \ref{lemma:the:time} and this choice verifies (A6) and (A9). Consider $L=e^{s+t}$ and $c=e^{-r}$. Since $|s|<3$, by Lemma \ref{lem:A0 to A4} assumptions (A0-4) hold for this $L$ and $c$. Indeed, $C_1,C_2,C_3=M$ and $C_4$ is $e^{-3}$ times the minimum of the shortest distance between the marked points taken over surfaces in $\mathcal{K}$.  By Lemma \ref{lem:cont in cpct} (and the fact that the projection of $\mathcal{K}$ to $\mathcal{M}_1$ is compact) (A7) holds. Indeed if $\hat{C}$ is the projection of $\mathcal{K}$ to $\mathcal{M}$ and $\frac{c}{99}r>\ell_{\hat{C}}(u+2)$ (where $\ell_{\hat{C}}$ is as in Lemma \ref{lem:rot divergence}) then $L>q_{k+u}$. Moreover, by Corollary \ref{cor:happy times} for each $\eta>0$ there exists $\hat{c}_\eta$ so that (A5) holds. We now just need to show (A8) holds. If $e^{s+t}\in \{q_{f(t)},q_{f(t)+1}\}$ this is by Lemma \ref{lemma:DK}. Otherwise note $\psi_{nq_i}$ is at most $5+2n$ valued by Lemma \ref{lemma:happy:times}. By Lemma \ref{lem:in cpct} there exist $N_{r,\epsilon}$ so that $a_{f(t)+1}<N_{r,\epsilon}$ and thus $e^{s+t}=\ell q_j$ for some $\ell<N_{r,\epsilon}$.  We obtain  (A8) with $k_{e^{-r}}=5+2N_{r,\epsilon}$. 
\appendixmode

\section{The Sarnak conjecture and joinings of powers}
\label{sec:appendix:A}

The following result is a trivial modification of a note \cite{Harper:note}
of Harper, which is included for completeness. What is below is a lightly edited version of his note. See that note for connections with the work of other authors.
\begin{theorem}
Let $(X, T)$ be a topological dynamical system. Assume that there exists $C>1$ so that for every $n$, the set $B_n=\{m<n:T^m \text{ is not disjoint from }T^n\}$ has the property that if $m>m'\in B_n$ then $\frac m {m'}>C$ then $T$ is disjoint from M\"obius. Indeed for any continuous compactly supported function with integral 0, $F$, we have $\sum_{n=1}^M \mu(n)F(T^nx)=o(M)$. 
\end{theorem}
Let $\mu$ denote the M\"obius function. 

To prove Theorem 1 we shall require a lemma concerning the additive function
$$\omega_\tau (n) := \underset{
p|n,
p\leq e^{
\frac 1 \tau}}{\sum}1.$$
\begin{lemma}
Define $\mu_\tau := \underset{
p\leq e
^{\frac 1 \tau}}{\sum}\frac 1 p$ and let $N$ be any natural number. Then we
have the following variance estimate:
$$\underset{
n\leq N}{\sum}
(\omega_\tau (n) - \mu_\tau )^
2 \leq N\mu_\tau + O(e^{
\frac 1 \tau }).$$
\end{lemma}

Lemma 1 is a special case of the Tur\'{a}n-Kubilius inequality, but since the proof
is just a short calculation we shall give it in full. Expanding the sum in the statement we obtain
$$ \underset{
p,q\leq e
^{\frac 1 \tau}}
{\sum}
\underset{n\leq N}{\sum}
1_{p,q|n} - 2\mu_\tau
\underset{
p\leq e^{
\frac 1 \tau}}{\sum}
\lfloor \frac N p \rfloor + N\mu_\tau^ 2
$$
and on removing the square brackets, and paying attention to the diagonal contribution in the double sum, we see that is at most
$$\underset{
p,q\leq e^
{\frac 1 \tau}}{\sum}
[N/pq] - N\mu_\tau^2
 + N\mu_\tau + 2\mu_\tau \pi(e
^{\frac 1 \tau}),$$

\subsection{Completion of proof}
Let $F(n)=F(T^nx)$ and in view of Lemma 1 and the Cauchy-Schwarz inequality, we have that 
\begin{multline}\label{eq:key estimate}|\sum_{n=1}^N \mu(n)F(n)|=\\ |\frac 1 {\mu_\tau} \sum_{n\leq N}\mu(n)F(n) \sum_{p |n, p\leq e^{\frac 1 \tau}}1 +\sum_{n=1}^N\mu(n)F(n) \frac{\mu_\tau-\omega_\tau(n)}{\mu_\tau}|\\\leq \frac 1 {\mu_\tau}|\sum_{n=1}^N \mu(n)F(n)\sum_{p|n,p\leq e^{\frac 1 \tau}}1|+\sqrt{\frac{N(N\mu_\tau+O(e^{\frac 1 \tau}))}{\mu_\tau^2}}.
\end{multline}

Observe that for each $\tau$ there exists $N_0$ so that for all $N>N_0$ we have $\sqrt{\frac{N(N\mu_\tau+O(e^{\frac 1 \tau}))}{\mu_\tau^2}}<\sqrt{N}\sqrt{\frac{2N}{\mu_\tau}}=N\sqrt{\frac{2}{\mu_\tau}}.$

So now we control

$$ |\sum_{n\leq N} \mu(n)F(n) \sum_{p| n, p<e^{\frac 1 \tau}} 1|\leq |\sum_{n\leq N}\sum_{p|n,p<e^{\frac 1 \tau}}\mu(p)\mu(\frac n p)F(n)|+\sum_{n\leq N}2|\{p<e^{\frac 1 {\tau}}:p^2|n\}|\cdot \|F\|_{\sup} .$$
Because  $\sum_{n\leq N}2|\{p<e^{\frac 1 {\tau}}:p^2|n\}|$ is $O(N)$ we focus on the other term,  
\begin{equation}\label{eq:appendix bound}| \sum_{n\leq N}\sum_{p|n,p<e^{\frac 1 \tau}}\mu(p)\mu(\frac n p)F(n)|\leq \sum_{j\leq \log_2(N)} \sum_{2^j\leq k<2^{j+1}}| \mu(k)\sum_{p\leq \min\{e^{\frac 1 \tau},\frac Nk\}} \mu(p)F(pk)|.
\end{equation} We apply Cauchy-Schwartz to $ \sum_{2^j\leq k<2^{j+1}}| \mu(k)\sum_{p\leq \min\{e^{\frac 1 \tau},\frac Nk\}} \mu(p)F(pk)|$ and bound (\ref{eq:appendix bound}) by
\begin{multline*} \sum_{j\leq\log_2(N)}\sqrt{2^j\sum_{2^j\leq m<2^{j+1}}|\sum_{p\leq \min\{e^{\frac 1 \tau},\frac N m\}}\mu(p)F(pm)|^2}\leq \\ \sum_{j\leq \log_2(N)} \sqrt{2^j \sum_{p_1,p_2\leq \min\{e^{\frac 1 \tau}, \frac N {2^j}\}}  |\sum_{m=2^j}^{\min\{2^{j+1}, \frac N{p_1},\frac N {p_2}\}}F(p_1m)\overline{F(p_2m)}      |}.
\end{multline*}

  The contribution of the diagonal terms ($p_1=p_2$) is at most $2^j \sqrt{\pi (\min \{e^{\frac 1 \tau}, \frac N {2^j})\}} \|F\|_{\sup} $ where $\pi(n)$ is the number of primes less than or equal to $n$.
The contribution of the $p_1\neq p_2$ where $p_1$ and $p_2$ are not disjoint is at most 
$$2^j\sqrt{C \log (\min\{e^{\frac 1 \tau},\frac N {2^j}\})\pi (\min\{e^{\frac 1 \tau},\frac {N}{2^j})\}}\|F\|_{\sup}.$$ Summing over $j$ these terms give a contribution that is $O(N)$. 
Indeed, we estimate by  $\sum_{j\leq \log_2(N)} 2^j\sqrt{C\log(\pi(\frac{N}{2^j}))\pi(\frac{N}{2^j})}\leq \sum_{j=1}^k C 2^j O((k-j)+1)2^{\frac 1 2 (k-j)}$ for $k=\lceil \log_2(N)\rceil$. This is clearly $O(N)$.

For $\tau$ fixed we choose $M_0$ large enough so that for any $M>M_0$,  $p_1,p_2<e^{\frac 1 \tau}$ with $T^{p_1}$ disjoint from $T^{p_2}$, and $L\leq M$ we have $$|\sum_{n\leq L}F(p_1n)\overline{F(p_2n)}|<\tau M.$$ 
The contribution of the $p_1,p_2$ where $T^{p_1}$ and $T^{p_2}$ are disjoint and $2^j>M_0$ is at most
$$ \sqrt{2^j(\tau 2^j) \pi(\min \{e^{\frac 1 \tau},\frac N {2^j}\})^2}.$$ For fixed $\tau$, summing over $j$, this is also $O(N)$. Indeed we focus on $$\sum_{j:\frac{N}{2^j}<e^{\frac 1 \tau}}\sqrt{2^j(\tau 2^j) \pi(\min \{e^{\frac 1 \tau},\frac N {2^j}\})^2}$$ and observe that this is bounded by $O(N\tau \log(\frac 1 {\tau}))$. 
 If $N$ is large enough the terms when $2^j<M_0$ are also $O(N)$. Since $\mu_n \to \infty$ plugging this into the last line of 
(\ref{eq:key estimate}) and possibly choosing an even larger $N$ so that $\sqrt{\frac{N(N\mu_\tau+O(e^{\frac 1 \tau}))}{\mu_\tau^2}}<N\sqrt{\frac{2}{\mu_\tau}}$ completes the proof.

\section{Disjointness of powers for generic $3$-IET's}
\label{sec:appendix:B}
\begin{theorem}\label{thm:disjoint powers}For almost every $3$-IET, $T$ we have that $T^n$ is disjoint from $T^m$ for all $0<n<m$.
\end{theorem}
We prove this by the following straightforward disjointness criterion:
\begin{prop}\label{prop:criterion} Let $T$ be an ergodic 3-IET, $R$ be an irrational rotation and $0<n<m$ be natural numbers.
Assume there exists $c>0$, $r\in \mathbb{N}$, a sequence $k_1,...$ sets $F_i$, $G_i$ so that for all $i$
\begin{enumerate}
\item $\underset{i \to \infty}{\lim}\, \underset{x \in F_i}{\max} |T^{nk_i}x-x|=0$
\item $\underset{i \to \infty}{\lim}\, \underset{x \in G_i}{\max} |T^{mk_i}x-R^{-1}x|=0$
\item$1-\lambda(F_i)<\lambda(G_i)-c.$
\end{enumerate}
Then $T^n$ and $T^m$ are disjoint.
\end{prop}
\begin{proof}[Sketch of proof] Let $\sigma$ be an ergodic joining of $T^n \times T^m$ that is a probability measure. Because is $T$ is ergodic it suffices to show that $\sigma$ is $id \times T^{-1}$ invariant. By the fact that ergodic probability measures are mutually singular or the same it suffices to show that $(id \times R^{-1})_*\sigma$ is not singular with respect to $\sigma$.
By our assumptions, for any $i$ we have $\sigma(F_i \times G_i)\geq  c $. Similarly to Section 2, $\sigma$ is not singular with respect to $(id \times R^{-1})_* \sigma$.
\end{proof}

For any $\alpha$ let $\FP{q_j\alpha}=(-1)^j\|q_j\alpha\|$, the signed distance of $R^{q_j}x$ from $x$. If $x \in [0,1)$ there exists $b_1,...$ so that $b_i\leq a_i$, if $b_i=a_i$ then $b_{i+1}=0$ and  $x=\sum_{i=1}^{\infty}b_i\FP{q_{i-1}\alpha}$. 
  Notice that for any fixed $\alpha$ the set of $x$ with (an allowable) Ostrowski expansion $b_1,...,b_k$ is an interval of size at least $\fp{q_{k+1}\alpha}$. 
\begin{lemma}\label{lem:all ostrowski}Given a 3-IET consider it as rotation by $\alpha$ induced on an interval $[0,x)$. Let $[a_1,\dots]$ be the continued fraction of $\alpha$ and $(b_1,...)$ be the $\alpha$-Ostrowski expansion of $x$. For $\lambda^2$ almost every $(\alpha,x)$ we have that for any ordered k-tuple of pairs $(c_1,d_1),...,(c_k,d_k)$  of natural numbers so that $d_i\leq c_i-1$ we have that there are infinitely many $i$ with $((a_i,b_i),...,(a_{i+k-1},b_{i+k-1}))=((c_1,d_1),...,(c_k,d_k))$.
\end{lemma}
\begin{proof} For almost every $\alpha$ any $(k+1)$-tuple of natural numbers occurs infinitely often in its continued fraction expansion by the ergodicity of the Gauss map with respect to a fully supported finite invariant measure and the fact that having a fixed initial $(k+1)$-tuple $(c_1,\dots,c_{k+1})$ is a set of positive measure. For any $\alpha$ with this property, the set of $x$ so that the pair $(\alpha,x) $ satisfies the proposition is a set of full measure because the complement has no Lebesgue density points. Indeed let $\alpha$ have $a_{j+i}=c_i$ for $i\leq k+1$ and $y\in [0,1)$, then an interval of size at least  $\|q_{j+k+1}\alpha\|$ in $B(y,\|q_j\alpha|)$ have that the $j+1$ through $j+k$ terms of their Ostrowski expansion are $d_1,\dots d_{k-1}$. Since $\frac{\|q_{j+k+1}\alpha\|}{\|q_j\alpha\|}>3^{-(j+1)}c_1\cdots c_{k+1}$ we have the claim.
\end{proof}


\begin{proof}[Proof of Theorem \ref{thm:disjoint powers}]
By Lemma \ref{lem:all ostrowski} it suffices to show that any 3-IET given by inducing rotation by $\alpha$ on $[0,x)$ where the sequence $(a_1,b_1),...$ contains all $k$-tuples $(c_1,d_1),...,(c_k,d_k)$ with the condition that $c_i-1\geq d_i$ infinitely often satisfies the assumptions of Proposition \ref{prop:criterion} for some $c>0$. Let $(10m,0),(10m,0),(4m,1),(10m,0)$ be the pairs of $($continued fraction expansion, Ostrowski expansion$)$. Let $\ell$ be an index so that $(a_{\ell+1},b_{\ell+1}),(a_{\ell+2},b_{\ell+2}),(a_{\ell+3}, b_{\ell+3}),(a_{\ell+4},b_{\ell+4})=(10m,0),(10m,0),(4m,1),(10m,0)$ and $\ell+2$ is even (this can be done because the 8-tuple $$(10m,0),(10m,0),(4m,1),(10m,0),(10m,0),(4m,1),(10m,0),(10m,0)$$ occurs infinitely often). Let $x_{\ell-i}=\sum_{i=0}^{\ell-1}b_i\FP{q_i\alpha}$. Note there exists $r<q_\ell$ so that $x_{\ell-1}=R^{r}0.$ 

\noindent
\textbf{Sublemma:} There exists $j$ so that $\lambda(\{x:\sum_{i=0}^{q_{\ell+2}-1}\chi_{[0,x_{\ell-1})}(R^ix)\neq j\})<\frac 1 {100m^3}.$

Let $\phi_\ell(x)=\sum_{i=0}^{q_{\ell+2}-1}\chi_{[0,x_{\ell-1})}(R^ix)$
\begin{proof}If $\phi_\ell(x)\neq \phi_\ell (Rx)$ then $\chi_{[0,q_{\ell-1}(x))}(x)-\chi_{[0,x_{\ell-1})}(R^{q_k}x)\neq 0$. Since $\ell$ is even this means that $x \in [-\|q_{\ell+2}\alpha\|,0)\cup [x_{\ell-1}-\|q_{\ell+2}\alpha\|,x_{\ell-1}\|)$. Observe that $R^{r}([-\|q_{\ell+2}\alpha\|,0))=[x_{\ell-1}-\|q_{\ell+2}\alpha\|,x_{\ell-1})$. It follows that $\phi_\ell$ has two level sets, one on $\cup_{i=1}^{r} R^i([-\|q_{\ell+2}\alpha\|,0))$ and the other on its complement. $r\lambda([-\|q_{\ell+2}\alpha\|,0))\leq q_\ell \|q_{\ell+2}\|<\frac 1 {100m^3}$.
\end{proof}

Now consider $J=[0,x_{\ell-1}) \cup (J\setminus [0,x_{\ell-1})$. By the sublemma there exists $j$ so that $\lambda(\{x:\sum_{i=0}^{mq_{\ell+2}-1}\chi_{ [0,x_{\ell-1})}(R^ix)=mj\})>1-\frac m {100m^3}$ and  $\lambda(\{x:\sum_{i=0}^{nq_{\ell+2}-1}\chi_{ [0,x_{\ell-1})}(R^ix)=nj\})>1-\frac n {100m^3}>1-\frac m {100m^3}$. Now since $b_{\ell+3}=0$
\begin{multline*}
\fp{q_{\ell+2}\alpha}-\frac 1 {10m}\fp{q_{\ell+2}\alpha} <\fp{q_{\ell+2}\alpha}-\fp{q_{\ell+4}\alpha}\leq|(J\setminus [0,x_{\ell-1})|<\\
\fp{q_{\ell+2}\alpha}+\fp{q_{\ell+3}\alpha}<\fp{q_{\ell+2}\alpha}+\frac 1 {10m}\fp{q_{\ell+2}\alpha}
\end{multline*}
 and $R^i(J \setminus [0,x_{\ell-1})$ are disjoint for all $j<q_{\ell+3}$ we have 
$$\lambda(\{x:\sum_{i=0}^{mq_{\ell+2}-1}\chi_J(R^ix)=mj+1\})>mq_{\ell+2}\fp{q_{\ell+2}\alpha}-\frac 1 {100m^2}-mq_{\ell+2}\frac 1 {10m}\fp{q_{\ell+2}\alpha}$$ and 
$$\lambda(\{x:\sum_{i=0}^{nq_{\ell+2}-1}\chi_J(R^ix)=nj\})\geq 1-nq_{\ell+2}\fp{q_{\ell+2}\alpha}-\frac 1 {100m^2}-nq_{\ell+2}\frac 1 {10m}\fp{q_{\ell+2}\alpha}.$$
We have verified the assumptions of Proposition \ref{prop:criterion} with  $c=(m-n-\frac{m+n}{10m})\|q_{\ell+2}\alpha\|q_{\ell+2}-\frac{2}{100m^2}$. Since $q_{\ell+2}\|q_{\ell+2}\alpha\|>\frac{1}{10m+3}$ this is positive. 
\end{proof}

\end{document}